\documentclass[11pt]{article}

 \usepackage{amsmath,amssymb,amscd,amsthm,esint}

\usepackage{graphics,amsthm,mathrsfs,amsfonts}

\usepackage{sidecap}

\usepackage{float}

\usepackage{extarrows}

\usepackage{booktabs}

\usepackage{verbatim}

\usepackage[usenames,dvipsnames]{xcolor}

\usepackage{hyperref}

\newtheorem{thm}{Theorem}[section]
\newtheorem{lemma}[thm]{Lemma}
\newtheorem{cor}[thm]{Corollary}

\theoremstyle{definition}
\newtheorem{remark}[thm]{Remark}

\def\XXint#1#2#3{{\setbox0=\hbox{$#1{#2#3}{\int}$}
         \vcenter{\hbox{$#2#3$}}\kern-.5\wd0}}

\def\R{\mathbb{R}}
\def\C{\mathbb{C}}
\def\e{\varepsilon}

\def\loc{\text{loc}}

\def\A{\mathbf{A}}
\def\B{\mathbf{B}}

\def\P{\mathbf{P}}
\def\b{\mathcal{B}}
\def\e{\varepsilon}

\numberwithin{equation}{section}

\begin{document}

\title{Boundary Value Problems for the Magnetic Laplacian \\  in Semiclassical Analysis}

\author{
Zhongwei Shen 
 }
\date{}
\maketitle

\begin{abstract}

This paper is concerned with the magnetic Laplacian
$P^h  (\A)=(h D+\A)^2$ in semiclassical analysis, where $h$ is a semiclassical parameter.
We study the $L^2$ Neumann and Dirichlet problems for the equation $P^h(\A) u=0$ in a bounded Lipschitz domain $\Omega$.
Under the assumption that the magnetic field $\nabla \times \A$
is of finite type on $\overline{\Omega}$,
we establish the nontangential maximal function estimates for $(h D+\A)u$,
which are uniform for $0< h< h_0$.
This extends a well-known result due to D. Jerison and C. Kenig for the Laplacian in Lipschitz domains
to the magnetic Laplacian in the semiclassical setting.
Our results  are new even for smooth domains.

\medskip


\medskip

\noindent {\it MR (2020) Subject Classification}: 35P25.

\end{abstract}


\section{Introduction}\label{section-1}

Let  $D=-i \nabla$ and $\A=(A_1, A_2, \dots, A_d)\in C^1(\R^d; \R^d)$. Consider the Schr\"odinger operator with a magnetic potential
in semiclassical analysis,
\begin{equation}\label{op-1}
P^h  (\A) =(h D+\A)^2, 
\end{equation}
also called the magnetic Laplacian, 
where $h\in (0, 1) $ is a semiclassical  parameter. In this paper, we are interested in boundary regularities of  $P^h (\A)$ that are
uniform in $h$.
More precisely, consider the Neumann problem in a bounded Lipschitz domain $\Omega$,
\begin{equation}\label{NP0}
\left\{
\aligned
P^h (\A)  u& =0 & \quad & \text{ in } \Omega,\\
n \cdot (hD+\A) u& =g & \quad & \text{ on } \partial\Omega,
\endaligned
\right.
\end{equation}
where $n=(n_1, n_2, \dots, n_d)$ denotes the outward unit normal to $\partial\Omega$, and the Dirichlet problem,
\begin{equation}\label{DP0}
\left\{
\aligned
P^h (\A) u  & =0 & \quad & \text{ in } \Omega,\\
u & =f & \quad & \text{ on } \partial\Omega.
\endaligned
\right.
\end{equation}
Let $\B=\nabla \times \A$ denote the magnetic field. In the main results of this paper, 
we shall assume that $\B$ is of finite type on $\overline{\Omega}$; i.e., 
 there exist an integer $\kappa\ge 0$ and $c_0>0$ such that 
\begin{equation}\label{MH}
\sum_{|\alpha|\le \kappa}
|\partial^\alpha \B (x)|\ \ge c_0
\end{equation}
for any $x\in \overline{\Omega}$.
Recall that the nontangential maximal function is defined by 
\begin{equation}\label{max}
(u)^*(x)=\sup
\big\{   |u(y)|: \, y\in \Omega \text{ and } \ |y-x|< M_0\,  \text{\rm dist} (y, \partial\Omega) \big\}
\end{equation}
for $x\in \partial\Omega$,
where $M_0>1$ is a large (fixed) constant depending on $\Omega$.
The following two theorems are the main results of this paper.

\begin{thm}\label{main-thm-1}
Let $\Omega$ be a bounded Lipschitz domain in $\R^d$, $d\ge 2$.
Let $\A \in C^\infty(\overline{\Omega}; \R^d)$.
Suppose $\B$ is of finite type on $\overline{\Omega}$.
Then for any $g\in L^2(\partial\Omega, \C)$, a solution  $u\in H^1(\Omega; \C)$ to the Neumann problem  \eqref{NP0}
satisfies the nontangential maximal function estimate,
\begin{equation}\label{main-e1}
\| (v_h)^*\|_{L^2(\partial\Omega)}
\le C \| g \|_{L^2(\partial\Omega)}
\end{equation}
for $h\in (0, h_0)$, where $v_h(x)=|(hD+\A) u(x)| + h\,  m(x, h^{-1} \B) |u(x)|$.
Moreover, we have
\begin{equation}\label{main-e2}
\aligned
 & \int_\Omega m(x, h^{-1}\B)  |v_h(x) |^2 dx
  \le C \int_{\partial\Omega} |g|^2.
 \endaligned
\end{equation}
The constants $C$ in \eqref{main-e1}-\eqref{main-e2}  and $h_0$  depend only 
on $\Omega$, $(\kappa, c_0)$ in \eqref{MH}, and $\|\B \|_{C^{\kappa+1}(\overline{\Omega})}$.
\end{thm}

Let   $T^h=(T^h_{jk})$, where $1\le j, k\le d$ and
\begin{equation*}
T^h_{jk}
=n_j (hD_k+A_k) - n_k (h D_j +A_j)
\end{equation*}
 is a (tangential)  differential operator on $\partial\Omega$.

\begin{thm}\label{main-thm-2}
Assume that $\Omega$ and $\B$ satisfy the same conditions as in Theorem \ref{main-thm-1}.
Then for any $f\in H^1(\partial\Omega; \C)$, a solution $u \in H^1(\Omega; \C)$ 
to the Dirichlet problem \eqref{DP0} satisfies the estimate,
\begin{equation}\label{main-e3}
\aligned
 & \| (v_h)^*\|_{L^2(\partial\Omega)}
 \le C \Big\{
\|  T^h f\|_{L^2(\partial\Omega)}
+ h  \|  m(x, h^{-1} \B) f \|_{L^2(\partial\Omega)} \Big\}
\endaligned
\end{equation}
for $h\in (0, h_0)$, where $v_h(x)=|(hD+\A) u(x)| + h\,  m(x, h^{-1} \B) |u(x)|$.
Moreover, we have
\begin{equation}\label{main-e4}
\aligned
 & \int_\Omega m(x, h^{-1} \B)  |v_h (x)|^2  dx
 \le C \Big\{
\|  T^h f\|_{L^2(\partial\Omega)}^2
+  h^2 \|  m(x, h^{-1} \B) f \|_{L^2(\partial\Omega)}^2 \Big\}.
 \endaligned
\end{equation}
The constants $C$ in \eqref{main-e3}-\eqref{main-e4}  and $h_0$  
depend only on $\Omega$, $(\kappa, c_0)$ in \eqref{MH}, and $\|\B \|_{C^{\kappa+1}(\overline{\Omega})}$.
\end{thm}

A  few remarks are in order.

\begin{remark}
{\rm

In the case $\A \equiv 0$ in $\Omega$ and $h=1$, the Neumann and Dirichlet problems \eqref{NP0} and \eqref{DP0} reduce to
\begin{equation}\label{NH0}
\Delta u =0 \quad \text{ in } \Omega \quad \text{ and } \quad \frac{\partial u}{\partial n} =g \quad \text{ on } \partial \Omega,
\end{equation}
and 
\begin{equation}\label{DH0}
\Delta u =0 \quad \text{ in } \Omega \quad \text{ and } \quad  u = f\quad \text{ on } \partial \Omega,
\end{equation}
respectively.
In this case, it was proved by D. Jerison and C. Kenig \cite{DK-1980, DK-1981} that the weak solutions in $H^1(\Omega)$
to \eqref{NH0} and \eqref{DH0}
satisfy the nontangential maximal function estimates, 
$$
\| (\nabla u)^*\|_{L^2(\partial\Omega)}\le C \| g \|_{L^2(\partial\Omega)}
\quad \text{ and } \quad
\| (\nabla u)^* \|_{L^2(\partial\Omega)} \le C \| f \|_{H^1(\partial\Omega)},
$$
respectively, assuming that $\Omega$ is a bounded Lipschitz domain.
Our main results extend this classical work to the magnetic Laplacian $P^h(\A)$
in the semiclassical setting, where the estimates are uniform in $h\in (0, h_0)$.
To the best of the author's knowledge,
Theorems \ref{main-thm-1} and \ref{main-thm-2} are new
even for smooth domains.
}
\end{remark}

\begin{remark}\label{rem-1}
{\rm

 The   auxiliary function $m(x, \B)$ in Theorems \ref{main-thm-1} and \ref{main-thm-2} is  defined by
\begin{equation}\label{m}
\frac{1}{m(x, \B)}
=\sup \Big\{ r>0:
\sup_{\b(x, r)} |\B| \le  \frac{1}{r^2}  \Big\},
\end{equation}
where $\b(x, r)$ denotes the ball centered at $x$ with radius $r$.
This function, which was introduced by the present author in \cite{Shen-1994, Shen-1995}, plays the role  of  a
critical  scaling in the study of Schr\"odinger operators with electrical and magnetic potentials.
See earlier work in \cite{HN-1985, smith-1991, Zhong-1993} as well as references for more recent work in \cite{Poggi-2024, Shen-2025}.
Note that if $\B$ is a (matrix-valued) polynomial of degree $\kappa$, then
$$
\sup_{\b(x, r)} |\B| \approx \sum_{|\alpha|\le \kappa} |\partial^\alpha \B(x)| r^{|\alpha|}.
$$
It follows that 
\begin{equation*}
 c \sum_{|\alpha|\le \kappa}
 |\partial^\alpha \B (x)|^{\frac{1}{|\alpha|+2}}
 \le m(x,\B) \le C 
\sum_{|\alpha|\le \kappa}  | \partial^\alpha \B (x)|^{\frac{1}{|\alpha|+2}},
\end{equation*}
where $C, c>0$ depend only on $d$ and $\kappa$.
Under the finite-type condition on $\B$ in Theorem \ref{main-thm-1}, we have 
 $$
m (x, h^{-1} \B) \ge c  \sum_{|\alpha|\le \kappa}
 h^{-\frac{1}{|\alpha|+2}} |\partial^\alpha \B (x)|^{\frac{1}{|\alpha|+2}}
$$
 for $h\in (0, h_0)$. See Remark \ref{re-9-1}.
}
\end{remark}

\begin{remark}\label{re-dn}
{\rm

Consider the Dirichlet-to-Neumann map associated with the operator $P^h(\A)$,
\begin{equation*}
\Lambda^h: f \to n \cdot (h D +\A )u,
\end{equation*}
where $u$ is the solution of the Dirichlet problem \eqref{DP0}.
It follows from Theorems \ref{main-thm-1} and \ref{main-thm-2}  that 
\begin{equation}\label{DN-1}
\| \Lambda ^h f \|_{L^2(\partial\Omega)}
\approx
\|  T^h f\|_{L^2(\partial\Omega)}
+ h  \|  m(x, h^{-1} \B) f \|_{L^2(\partial\Omega)} 
\end{equation}
for $h\in (0, h_0)$. Note that $T^hf $ represents the tangential component of $(h D+\A)u$ on $\partial\Omega$.
In fact, our proof yields the weighted estimates,
$$
\| m(x, h^{-1}\B)^\ell \Lambda^h f \|_{L^2(\partial\Omega)}
\approx
\| m(x, h^{-1} \B)^\ell T^h f \|_{L^2(\partial\Omega)}
+ h \| m(x, h^{-1}\B)^{\ell+1} f \|_{L^2(\partial\Omega)}
$$
for any $\ell \in \R$ (the bounding constants depend on $\ell$).
See Theorem \ref{re-theorem}.
}
\end{remark}

\begin{remark}
{\rm
Let $u$ be a solution of the Dirichlet problem \eqref{DP0} with $f\in C(\partial\Omega; \C)$. 
Then
\begin{equation}\label{D-est}
\| (u)^* \|_{L^2(\partial\Omega)}
\le C \| f \|_{L^2(\partial\Omega)}.
\end{equation}
This follows  readily from the maximum principle for subharmonic functions.
No addition condition beyond $\A\in C^1(\overline{\Omega}; \R^d)$  is not needed.
See Theorem \ref{thm-p1}.
}
\end{remark}

\begin{remark}
{\rm

Suppose $\B$ is of finite type on $\partial\Omega$; i.e., the inequality \eqref{MH} holds
for any $x\in \partial\Omega$.
By compactness it follows that \eqref{MH} holds in a neighborhood of $\partial\Omega$.
In this case, the estimates
\eqref{main-e1} and \eqref{main-e3} hold if one modifies the definition of the nontangential maximal function $(u)^*$
by considering only points near $\partial\Omega$,
\begin{equation*}
(u)^* (x)=\sup\big \{ |u(y)|: \ y\in \Omega,\ |y-x|< M_0 \text{\rm dist} (y, \partial\Omega) \text{ and } \text{\rm dist}(y, \partial\Omega)< c\big \}.
\end{equation*}
This follows directly  from the proofs of Theorems \ref{main-thm-1} and \ref{main-thm-2}.
}
\end{remark}

We now describe the main ingredients  in the  proofs of Theorems \ref{main-thm-1} and \ref{main-thm-2}.
First, we rewrite the equation ${P}^h (\A) u=0$ as 
$(D+\beta \A)^2 u=0$, where $\beta =h^{-1}>1$.
By a localization argument, we  reduce the problem to the equation 
$(D+\A)^2 u=0$ in a bounded Lipschitz domain $\Omega$, assuming that $\Omega \subset \b(0, R_0)$ for some $R_0>0$
and  the magnetic field $\B$ satisfies  the following conditions.
There exists $C_0>0$ such that
\begin{equation}\label{H0}
\sup_{\b(x, r)} |\nabla \B| \le \frac{C_0}{r} \fint_{\b(x, r)} |\B|
\end{equation}
for any $\b(x, r) \subset \b(0, 4R_0)$,  and 
\begin{equation} \label{H1}
\sup_{\b(x_0, R_0)} |\B|\ge 4R_0^{-2}
\end{equation}
for any $x_0\in \overline{\Omega}$.
Note that when $\B$ is replaced by $\beta \B$, 
the condition \eqref{H0} remains invariant, while \eqref{H1} is satisfied for $\beta$ sufficiently  large.
Moreover, under the conditions \eqref{H0}-\eqref{H1}, it is proved in  \cite{Shen-2025} that 
\begin{equation}\label{o-low}
 \int_\Omega m(x, \B)^2 |\psi|^2 dx \le C  \int_\Omega | (D+\A)\psi|^2 dx
 \end{equation}
 for any $\psi\in C^1(\overline{\Omega}; \C)$.

Next, let $u\in H^1(\Omega; \C)$ be a weak solution of $(D+\A)^2 u=0$ in $\Omega$.
By an approximation argument, we may assume that  $\Omega$ and thus $u$ are smooth.
To relate the $L^2$ norm of the normal component of $(D+\A) u$ on $\partial\Omega$
to the $L^2$ norm of its tangential components, we use two Rellich-type identities,
adapted to the operator $(D+\A)^2$.
See Lemmas \ref{lemma-r1} and \ref{lemma-r2}.
As a result, the proof for \eqref{DN-1} with $h=1$ is reduced to the estimates, 
\begin{equation}\label{0-e1}
\left\{
\aligned
\int_\Omega m(x, \B)^3 |u|^2\,  & dx \le C \int_{\partial\Omega}  m(x, \B)^2  |u|^2\,  dx, \\
\int_\Omega m(x, \B)^3 |u|^2\,  & dx \le C \int_{\partial\Omega} |n \cdot (D+\A)  u|^2\,  dx.
\endaligned
\right.
\end{equation}
To establish \eqref{0-e1}, we study the Green function $G_\A(x, y)$
 and the Neumann function $N_\A(x, y)$ for the operator $(D+\A)^2$
in $\Omega$. Suppose $d\ge 3$ (the case $d=2$ is handled by the method of descending). 
Using \eqref{o-low}, we are able to show that under the assumptions \eqref{H0}-\eqref{H1},
\begin{equation}\label{GN}
\left\{
\aligned
|G_\A(x, y)| & \le \frac{C_\ell}{ \{ 1+ |x-y| m(x, \B) \}^\ell} \cdot \frac{1}{|x-y|^{d-2}}, \\ 
|N _\A(x, y)| & \le \frac{C_\ell}{ \{ 1+ |x-y| m(x, \B) \}^\ell} \cdot \frac{1}{|x-y|^{d-2}}\\ 
\endaligned
\right.
\end{equation}
for any $\ell\ge 0$ and $x, y\in \Omega$.
See Sections \ref{sec-G} and \ref{sec-N}.
The estimates in \eqref{0-e1} follow from \eqref{GN} by using the Rellich identities mentioned above and a duality argument.

Finally, with the Rellich estimates \eqref{DN-1} at our disposal, we establish the nontangential maximal function 
estimates \eqref{main-e1} and \eqref{main-e3} with the help of \eqref{D-est}. 
A key observation here is   that although the commutator $[(D+\A)^2, D_k+A_k] $ may not be zero in $\Omega$,
 the function $v(x)=|(D+\A) u(x) | +m(x, \B) |u(x)|$ nevertheless satisfies
the interior estimate,
\begin{equation*}
v(x)
\le C \fint_{\b(x, r)}  v , 
\end{equation*}
if $\b(x, r)\subset \Omega$. See Section \ref{section-I}.
This allows us to use the techniques developed for \eqref{0-e1} to control the additional error terms. 


\section{Preliminaries}\label{section-P}

Throughout the section we assume that $\Omega$ is a bounded Lipschitz domain and $\A\in C^1({\Omega}; \R^d)$.

\begin{lemma}\label{lemma-p1}
Let $\e>0$ and $u_\e =\sqrt{ |u|^2 +\e^2}$.  Then for $u\in C^1(\Omega; \C)$, 
\begin{equation}\label{p1-1}
|\nabla u_\e| \le |(\nabla +i \A) u|,
\end{equation}
and for $u\in C^2(\Omega; \C)$,
\begin{equation} \label{p1-2}
u_\e \Delta u_\e
\ge \Re \left\{ (\nabla +i \A)^2 u \cdot \overline{u} \right\}.
\end{equation}
\end{lemma}

\begin{proof}

By differentiating  $u_\e^2 = |u|^2 +\e^2$, we obtain
\begin{equation}\label{p1-3}
u_\e \partial_j u_\e
= \Re \{ (\partial_j +i A_j) u\cdot  \overline{u} \}
\end{equation}
for $1\le j\le d$.
Hence, $u_\e |\nabla u_\e|\le |(\nabla + i \A)u| |u|$.
Since $|u|\le u_\e$, this yields \eqref{p1-1}.

Next, by differentiating \eqref{p1-3}, we see that 
\begin{equation}\label{p1-4}
|\nabla u_\e|^2 + u_\e \Delta u_\e
=\Re \{ (\nabla  + i \A)^2 u \cdot \overline{u} \}
+ |(\nabla + i \A) u|^2,
\end{equation}
which, together with \eqref{p1-1} and $|u|\le u_\e$, leads to \eqref{p1-2}.
\end{proof}

\begin{lemma}\label{lemma-P}
Let $u\in H_0^1(\Omega; \C)$. Then
\begin{equation}\label{P-1}
\left(  \int_\Omega |u|^{p_d} \right)^{1/p_d} \le C  \left( \int_\Omega |(D +\A)u |^2 \right)^{1/2},
 \end{equation}
 where $p_d=\frac{2d}{d-2}$  for $d\ge 3$, 
 $2< p_d< \infty$ for $d=2$, and  $C$ depends  only on $\Omega$ and $p_d$.
 Furthermore, if $u\in H^1(\Omega; \C)$, then
\begin{equation}\label{P-1a}
\left(  \int_\Omega |u|^{p_d} \right)^{1/p_d} \le C  \left( \int_\Omega |(D +\A)u |^2 \right)^{1/2}
+ C \left(\int_\Omega |u|^2 \right)^{1/2},
\end{equation}
and
\begin{equation}\label{P-1c}
\int_{\partial\Omega} |u|^2
\le C \int_\Omega |(D+\A) u|^2 + C \int_\Omega |u|^2.
\end{equation}
\end{lemma}

\begin{proof}

 Let $v_\e =\sqrt{|u|^2 + \e^2} -\e$, where $\e>0$.
It follows from  \eqref{p1-1} that 
$|\nabla v_\e |\le  |(D+\A) u|$.
Since $v_\e \in H_0^1(\Omega; \R)$, by Sobolev  inequality,
$$
\left( \int_\Omega |v_\e|^{p_d} \right)^{1/p_d}
\le C\left(  \int_\Omega |\nabla v_\e|^2\right)^{1/2}
\le C \left(\int_\Omega | (D+\A) u|^2\right)^{1/2}.
$$
By letting $\e \to 0$, we obtain \eqref{P-1} for $u\in C_0^1(\Omega, \C)$.
A density argument gives \eqref{P-1} for $u\in H_0^1(\Omega; \C)$.
The proofs for \eqref{P-1a} and \eqref{P-1c}  are similar.
\end{proof}

\begin{thm}\label{thm-P}
Let $u\in C^2(\b(x_0, 2r); \C)$ be a solution of $(D +\A)^2 u=F$ in $\b (x_0, 2r)$.
Then
\begin{equation}\label{G2}
\left(\fint_{\b(x_0, r)} |u|^q\right)^{1/q}
\le C \fint_{\b(x_0, 2r)} |u| 
+ C r^2 \left(\fint_{\b(x_0, 2r)} |F|^p \right)^{1/p},
\end{equation}
where $2\le p< q \le \infty$ and $\frac{1}{p}-\frac{1}{q}< \frac{2}{d}$.
The constant $C$ depends only on $d$, $p$ and $q$,
\end{thm}

\begin{proof}

We give the proof for $d\ge 3$.
The case $d=2$ is similar.
Let $u_\e=\sqrt{|u|^2 +\e^2}$, where $\e>0$.
By \eqref{p1-2}, we have $\Delta u_\e \ge -|F|$ in $\b(x_0, 2r)$.
By differentiating 
$
\fint_{\partial \b(y, t)} u_\e
$
in $t$, as in the proof of the mean value property for harmonic functions, 
it  follows that  if $\b(y, t)\subset \b(x_0, 2r)$,
$$
u_\e (y)
\le \fint_{\partial \b(y, t)} u_\e
+ C \int_{\b(y, t)} \frac{ |F(z)|}{|z-y|^{d-2}} dz.
$$
Hence, for $y\in \b(x_0, r)$,
$$
u_\e (y)
\le  \fint_{\b (y, r)} u_\e 
+ C \int_{\b(y, r)} \frac{|F(z)|}{|z-y|^{d-2}} dz.
$$
By letting $\e \to 0$, we obtain
$$
|u (y)|
\le  \fint_{B(y, r)} |u|
+ C \int_{\b(y, r)} \frac{|F(z)|}{|z-y|^{d-2}} dz
$$
for any $y\in \b(x_0, r)$.
This yields \eqref{G2} by using Young's inequality.
\end{proof}

\begin{lemma}\label{lemma-p2}
Let $u\in C^2(\Omega; \C)\cap C(\overline{\Omega}; \C)$ be a solution of the Dirichlet problem,
\begin{equation}\label{p2-1}
(D+\A)^2 u =0 \quad \text{ in } \Omega \quad \text{ and } \quad
u= f\quad \text{ on } \partial\Omega,
\end{equation}
where $f\in C(\partial\Omega; \C)$. Then $|u|\le v$ in $\Omega$, where $v$ is harmonic in $\Omega$ and
$v=|f| $ on $\partial\Omega$.
\end{lemma}

\begin{proof}
Let  $\e>0$ and $w_\e = u_\e-v $, where $u_\e =\sqrt{ |u|^2 +\e^2}$.
It follows from \eqref{p1-2} that $\Delta w_\e=\Delta u_\e \ge  0$ in $\Omega$.
Thus, $w_\e$ is subharmonic in $\Omega$.
Note that
$$
w_\e =\sqrt{|f|^2+\e^2} -|f|\le \e \quad \text{ in } \partial \Omega.
$$
By the maximum principle for subharmonic functions, this implies that 
$$
\max_{\overline{\Omega}} w_\e= \max_{\partial\Omega} w_\e \le \e.
$$
Hence, $u_\e \le v +\e$ in $\Omega$.
By letting $\e\to 0$, we obtain $|u|\le v$ in $\Omega$.
\end{proof}

\begin{thm}\label{thm-p1}
Let $\Omega$ be a bounded Lipschitz domain and $\A\in C^1(\Omega; \R^d)$.
Let $u\in C^2(\Omega; \C)\cap C (\overline{\Omega}; \C)$ be a solution of the Dirichlet problem \eqref{p2-1}. Then
\begin{equation}\label{p3-1}
\| (u)^* \|_{L^p(\partial\Omega)}
\le C \| f \|_{L^p(\partial\Omega)}
\end{equation}
for $2-\delta< p \le \infty$, where $\delta>0$ depends on $\Omega$.
The constant $C$ in \eqref{p3-1} depends only on $d$, $p$, the Lipschitz character of  $\Omega$.
If $\Omega$ is $C^1$, the estimate \eqref{p3-1} holds for $1< p\le \infty$.
\end{thm}

\begin{proof}
By Lemma \ref{lemma-p2}, we see that $(u)^* \le (v)^*$ on $\partial\Omega$, where $v$ is a harmonic function in 
$\Omega$  such that $v=|f|$ on $\partial\Omega$.
As a consequence, the estimate \eqref{p3-1} follows from the well-known results for harmonic functions in Lipschitz and $C^1$
domains \cite{Kenig-book}.
\end{proof}


\section{The function $m(x, \B)$}

Let $\Omega$ be a bounded Lipschitz domain in $\R^d$, $d\ge 2$. It follows that there exist $r_0>0$ and $M_0>0$ such that 
for any $x_0\in\partial\Omega$, 
\begin{equation}
 \Omega\cap \b(x_0, r_0)
=\left\{ (x^\prime, x_d)\in \R^d: x^\prime \in \R^{d-1} \text{ and } x_d > \phi (x^\prime) \right\} \cap \b(x_0, r_0)
\end{equation}
in a new coordinate system, obtained from the standard one through translation and rotation,
where $\phi: \R^{d-1} \to \R$ is a Lipschitz function with $\|\nabla \phi \|_\infty \le M_0$ and $\phi (x_0)=0$. 
A constant $C$ is said to depend on the Lipschitz character of $\Omega$ if $C$ depends on $M_0$ and the number of balls
$\b(x_0, r_0)$,  centered on $\partial\Omega$, which are needed to cover $\partial \Omega$.
By translation, without the loss of generality,  we assume that $\Omega\subset \b(0, R_0)$ for some  $R_0=C r_0>0$, 
where $C$ depends on the Lipschitz character of $\Omega$.

We will impose the following conditions on $\B$ in this and next few sections: 
$\B \in C^1(\b(0, 4R_0); \R^{d\times d})$ and there exists $C_0>0$ such that
\begin{equation}\label{cod-1}
\sup_{\b(x, r)} |\nabla \B|
\le \frac{C_0}{r} \fint_{\b(x, r)} |\B|
\end{equation}
for any ball $\b(x, r) \subset \b(0, 4R_0)$.
 It follows from \eqref{cod-1} that
\begin{equation}\label{cod-2}
\sup_{\b } |\B |\le  C_1 \fint_{\b} |\B|
\end{equation}
for any ball $\b\subset \b(0, 4R_0)$; i.e., 
 $|\B|$ is a $B_\infty$ weight in $\b(0, 4R_0)$. In particular, $|\B|$ satisfies the doubling condition,
 \begin{equation}\label{dc}
 \int_{2\b} |\B| \le C \int_{\b} |\B|
 \end{equation}
 for any $2\b\subset \b(0, 4R_0)$, where $C$ depends only on $d$ and $C_1$ in \eqref{cod-2}.
 As a consequence, we also have
 \begin{equation}\label{cod-2a}
 \sup_{Q(x, r)} |\B|\le C \fint_{Q(x, r)} |\B|
 \end{equation}
if  $Q(x, r) \subset \b(0, 2R_0)$, where $Q(x,r)$ denotes the cube center at $x$ with side length $r$.

For $x\in \b(0, R_0)$, let $m(x, \B)$ be defined by \eqref{m}. 
Under the assumption that for any $x\in \b(0, R_0)$,
\begin{equation}\label{cod-1a}
\sup_{\b(x, R_0/2)} |\B|  > 4 R_0^{-2}, 
\end{equation}
it follows by definition that $m(x, \B)\ge   2R_0^{-1}$ for any $x\in \b(0, R_0)$.

\begin{lemma}\label{lemma-M1}
Suppose $\B$ satisfies \eqref{cod-1} for any ball $\b(x, r) \subset \b(0, 4R_0)$. 
Also assume that \eqref{cod-1a} holds for any $x\in \b(0, R_0)$. Then, 
 \begin{equation}\label{m-2d}
  |\B (x)|^{1/2} +|\nabla \B (x)|^{1/3}
  \le C m (x, \B)
  \end{equation}
  for any $x\in \b(0, R_0)$. Moreover, 
 for any $x, y\in \b (0, R_0)$,
\begin{equation}\label{m-2a}
m(y, \B) \le C \left\{ 1+ |x-y| m (x, \B) \right\}^{\kappa_0} m (x, \B),
\end{equation}
\begin{equation}\label{m-2b}
m(y, \B) \ge \frac{c\,  m (x, \B) }{ \{ 1+|x-y| m(x, \B) \}^{\frac{\kappa_0}{\kappa_0+1}} }.
\end{equation}
The constants $C, c, \kappa_0>0$ depend only on  $C_0$ in \eqref{cod-1}.
\end{lemma}

\begin{proof}

Let $r=\{ m(x, \B)\}^{-1}$. Using \eqref{cod-1a}, it is not hard to see that 
\begin{equation}\label{m-4a}
\sup_{\b(x, r)} |\B |= r^{-2}.
\end{equation}
This implies that $|\B(x)|\le r^{-2}=\{ m (x, \B)\}^2$.  
Also, by \eqref{cod-1} and \eqref{m-4a}, 
$$
|\nabla \B(x) |\le C_0 r^{-3} =C_0  \{ m (x, \B)\}^3.
$$
As a result, we have proved \eqref{m-2d}.

The inequalities \eqref{m-2a}-\eqref{m-2b} were proved by the present author in \cite{Shen-1995}
for any $x, y\in \R^d$, 
under the assumption that $|\B|$ is a $B_q$ weight in $\R^d$ for  some $q>(d/2)$.
With the conditions in the lemma, the same argument gives \eqref{m-2a}-\eqref{m-2b}
for any $x, y \in \b(0, R_0)$.
Note that the estimates \eqref{m-2d}-\eqref{m-2a} are scaling invariant;  
 the constants $C, c, \kappa_0$ do not depend on $R_0$.
\end{proof}

It follows readily from \eqref{m-2a}-\eqref{m-2b} that 
\begin{equation}\label{m-2c}
\aligned
c \left\{ 1+|x-y| m(y, \B) \right\}^{\frac{1}{\kappa_0 +1}}
&\le 1 + |x-y| m (x, \B)\\
& \le C\left \{ 1+ |x-y| m (y, \B) \right\}^{\kappa_0+1}
\endaligned
\end{equation}
for any $x, y\in \b(0, R_0)$.

\begin{thm}\label{thm-M1}
Suppose $\B$ satisfies \eqref{cod-1} for any ball $\b(x, r)\subset \b(0, 4R_0)$. 
Also assume that \eqref{cod-1a} holds for any $x\in \b(0, R_0)$.
  Then
\begin{equation}\label{M1-0}
c \int_\Omega \{ m (x, \B) \}^2 |\psi|^2\, dx
\le \int_{ \Omega} |(D+\A)\psi|^2\, dx 
\end{equation}
for any $\psi\in C^1 (\overline{\Omega}; \C)$,
where $c>0$ depends only on $C_0$ in \eqref{cod-1} and the Lipschitz character of $\Omega$.
\end{thm}

\begin{proof}

See \cite[Theorem 3.8]{Shen-2025}.
We point out that the proof in \cite{Shen-2025} only uses the fact $\B =\nabla \times \A$ in $\Omega$.
Outside of $\Omega$, there is no need to assume that $\B$ is a curl of some vector field.
This observation allows us to extend $\B$ to a neighborhood of $\Omega$ without extending $\A$
in the proof of our main results.
\end{proof}

The next lemma gives a Caccioppoli inequality for the operator $(D+\A)^2$. 
The conditions on $\B$ are not needed.

\begin{lemma}\label{lemma-Ca}
Let $x_0\in \overline{\Omega}$ and $0< r< r_0$, where either $x_0 \in \partial\Omega$ or $\b(x_0, r) \subset \Omega$.
Suppose  that  $u \in H^1(\b(x_0, r) \cap \Omega; \C)$ and $(D+\A)^2 u=F$ in $\b(x_0, r) \cap \Omega$, 
where $F\in L^2(\b(x_0, r) \cap \Omega; \C)$.
If $x_0\in \partial\Omega$, we also assume that either $u=0$ or $n \cdot (D+\A) u=0$ on $\b(x_0, r) \cap \partial\Omega$, 
where $n$ denotes the outward unit normal to $\partial\Omega$.
Then for $0< s< t< 1$, 
\begin{equation}\label{Ca-0}
\int_{\b(x_0, sr)\cap \Omega} |(D+\A) u|^2 \le \frac{C}{(t-s)^2 r^2}
\int_{\b(x_0, tr)\cap \Omega} |u|^2
+ C r^2  \int_{\b(x_0, tr)\cap \Omega} |F|^2,
\end{equation}
where $C$ depends only on $d$.
\end{lemma}

\begin{proof}

The proof is similar to the case of $\A\equiv 0$. Using
$$
\int_{ \Omega} (D+\A) u \cdot \overline{(D+\A) (u\varphi^2)}=\int_{ \Omega} F \cdot \overline{u \varphi^2}
$$
where $\varphi \in C_0^\infty(\b(x_0, r); \R)$,  we obtain 
\begin{equation}\label{Ca-1}
\int_\Omega |(D+\A) u|^2 | \varphi|^2
\le C \int_\Omega |u|^2 |\nabla \varphi|^2
+ C \int_\Omega |F| |u| |\varphi|^2.
\end{equation}
To finish the proof, we choose $\varphi \in C_0^\infty (\b(x_0, tr); \R)$ such that $\varphi =1$ in $\b(x_0, sr)$
and $|\nabla \varphi|\le C (t-s)^{-1} r^{-1}$.
\end{proof}


\begin{remark}\label{C-re}
{\rm
Let $x_0\in \partial\Omega$ and $0< r< r_0$.
Suppose that  $u\in C^2(\overline{\b(x_0, r)\cap \Omega}; \C)$ and
$(D+\A)^2 u=F$ in $\b(x_0, r)\cap \Omega$.
It follows from the proof of Lemma \ref{lemma-Ca} that 
\begin{equation}
\aligned
\int_{\b(x_0, r/2)\cap \Omega} | (D+\A) u|^2
 & \le \frac{C}{r^2}
\int_{\b(x_0, r) \cap \Omega} |u|^2
+ C r^2 \int_{\b(x_0, r) \cap \Omega} |F|^2\\
 &\qquad+ C \int_{\b(x_0, r) \cap \partial\Omega} | n \cdot (D+\A) u| |u|,
 \endaligned
\end{equation}
where $C$ depends only on $d$.
}
\end{remark}

\begin{thm}\label{thm-Ca-m}
Suppose $\B$satisfies the same conditions as in Theorem \ref{thm-M1}.
Let $x_0\in \overline{\Omega}$ and $0< r< r_0$, where either $x_0 \in \partial\Omega$ or $\b(x_0, r) \subset \Omega$.
Suppose $u\in H^1(\b(x_0,  r) \cap \Omega; \C)$ and $(D+\A)^2 u =0$ in $\b(x_0, r) \cap\Omega$.
If $x_0\in \partial\Omega$, we also assume that either $u=0$ on $\b(x_0, r) \cap \partial\Omega$ or
$n \cdot (D+\A) u=0$ on $\b(x_0, r) \cap \partial\Omega$. Then
\begin{equation}\label{m-m}
\int_{\b(x_0, r/2)\cap \Omega}
|u|^2 \le  \frac{C_\ell}{ \{ 1+ r m (x_0, \B) \}^\ell } \int_{\b(x_0, r)\cap \Omega} |u|^2
\end{equation}
for any $\ell\ge 1$, where $C_\ell$ depends on $\ell $, the Lipschitz character of $\Omega$,  and $C_0$ in \eqref{cod-1}.
\end{thm}

\begin{proof}

We may assume $r m (x_0, \B) >10$ for otherwise the estimate is trivial.
Let $\varphi \in C_0^\infty (\b(x_0, r) \cap\Omega; \R)$.
It follows from \eqref{M1-0} that 
\begin{equation}\label{mm-1a}
\aligned
\int_\Omega \{ m (x, \B) \}^2 |u \varphi|^2\, dx
&\le C \int_\Omega |(D+\A) (u \varphi)|^2\\
&\le C \int_\Omega |u|^2 |\nabla \varphi|^2\, dx,
\endaligned
\end{equation}
where we have used \eqref{Ca-1} for the last inequality.
Let $(1/2)< s< t< 1$.
Choose $\varphi\in C_0^\infty (\b(x_0, tr); \R)$ so that $\varphi=1$ in $\b(x_0, sr)$ and
$|\nabla \varphi|\le C (t-s)^{-2} r^{-2}$.
By \eqref{mm-1a} and \eqref{m-2b}, this gives
$$
\frac{ \{ m (x_0, \B)\}^2}{  \{ 1+ r m (x_0, \B)\}^{\frac{2 \kappa_0}{\kappa_0+1}}}
\int_{\b(x_0, sr)\cap \Omega} | u |^2
\le \frac{C}{(t-s)^2 r^2} \int_{\b(x_0, tr)\cap \Omega} | u |^2, 
$$
which leads to
\begin{equation}\label{mm-2}
\int_{\b(x_0, sr)\cap \Omega} | u |^2
\le 
\frac{C}{(t-s)^2
\{ 1+ r m (x_0, \B) \}^{\frac{2}{\kappa_0+1}}} \int_{\b(x_0, tr)\cap \Omega} | u |^2.
\end{equation}
The estimate \eqref{m-m} now follows by iterating \eqref{mm-2}.
\end{proof}


\section{Interior estimates}\label{section-I}

Suppose $(D+\A)^2 u=0$ in $\b(x_0, r)$.
It follows by Theorem \ref{thm-P} that 
\begin{equation}\label{I-1}
|u(x_0)|
\le C \fint_{\b(x_0, r)} |u|, 
\end{equation}
where $C$ depends only on $d$.
In this section we establish pointwise  estimates
for $|(D+\A) u|$  and $ m(\cdot, \B) |u|$ under the conditions \eqref{cod-1} and \eqref{cod-1a} on the magnetic field $\B$.

\begin{thm}\label{I-thm-1}
Assume $\B$ satisfies the condition \eqref{cod-1} for any ball $\b(x, r) \subset \b(0, 4R_0)$.
Also assume that \eqref{cod-1a} holds  for any $x\in \b(0, R_0)$.
Suppose $(D+\A)^2 u=0$ in  $\b(x_0, r)\subset\b(0, R_0)$.
Then for any $\ell\ge 0$, 
\begin{equation}\label{I-2-00}
|u(x_0) |
\le \frac{C_\ell }{ \{  1+ r m(x_0, \B)\}^\ell }
\fint_{\b(x_0, r)} |u| , 
\end{equation}
\begin{equation}\label{I-2-0}
m (x_0, \B) |u(x_0)|
\le \frac{C_\ell }{ \{  1+ r m(x_0, \B)\}^\ell }
\fint_{\b(x_0, r)} m(x, \B) |u(x)|  dx, 
\end{equation}
where $C_\ell $ depends  on $\ell $  and $C_0$ in \eqref{cod-1}.
\end{thm}

\begin{proof}

It follows from \eqref{I-1} and \eqref{m-m} that 
\begin{equation}\label{I-2-1}
\aligned
|u(x_0)| & \le C \left(\fint_{\b(x_0, r/4)} |u|^2 \right)^{1/2}
 \le \frac{C_\ell}{ \{ 1+ r  m(x_0, \B)\}^\ell }
 \left(\fint_{\b(x_0, r/2)} |u|^2 \right)^{1/2}\\
 &  \le \frac{C_\ell}{ \{ 1+ r  m(x_0, \B)\}^\ell }
 \fint_{\b(x_0, r)} |u|
\endaligned
\end{equation}
for any $\ell \ge 1$. To see \eqref{I-2-0}, note that 
by \eqref{m-2b},  
\begin{equation}\label{I-2-2}
m(x_0, \B) \le C \{ 1+ r m(x_0, \B)\}^{\frac{\kappa_0}{\kappa_0+1}} m(x, \B)
\end{equation}
 for any $x\in \b(x_0, r)$. This, together with \eqref{I-2-00}, gives \eqref{I-2-0}.
\end{proof}

\begin{thm}\label{I-thm-2}
Under the same assumptions as in Theorem \ref{I-thm-1}, we have
\begin{equation}\label{I-3-0}
|(D+\A) u(x_0)|
\le \frac{C}{r} \fint_{\b(x_0, r)} |u|, 
\end{equation}
where $C$ depends on $C_0$ in \eqref{cod-1}.
\end{thm}

\begin{proof}

Note that 
$$
\aligned
 & (D+\A)^2 (D_k+A_k)u\\
& =-[D_k+A_k, D_j +A_j] (D_j +A_j)u 
- (D_j+A_j )[D_k+A_k, D_j +A_j]u\\
&=-2 [D_k+A_k, D_j +A_j] (D_j +A_j)u 
-[ [D_j +A_j , [D_k+A_k, D_j +A_j]] u,
\endaligned
$$
where the repeated index $j$ is summed from $1$ to $d$.
It follows that $(D+\A)^2 (D_k+A_k) u=F$, where
$$
|F|\le 2 |\B| |(D+\A)u| +|\nabla \B| |u|.
$$
In view of  \eqref{G2}, this implies that if $\b(y, 2t)\subset \b(x_0, r/4)$, 
\begin{equation}\label{I-3-1}
\aligned
\left(\fint_{\b(y, t)} |(D+\A) u|^q \right)^{1/q}
 & \le C \fint_{\b(y, 2t)} |(D+\A) u|
  + C t^2 \left(\fint_{\b(y, 2t)} |\B|^p |(D+\A) u|^p \right)^{1/p}\\
&\qquad\qquad
+ C t^2 \left(\fint_{\b(y, 2t)} |\nabla \B|^p |u|^p \right)^{1/p}, 
\endaligned
\end{equation}
where $1\le   p< q \le \infty$ and $\frac{1}{p}-\frac{1}{q}< \frac{2}{d}$.
Using $|\B (x)|\le C m(x, \B)^2$ and $|\nabla \B (x) |\le C m(x, \B)^3$, we obtain 
$$
\aligned
\left(\fint_{\b(y, t)} |(D+\A) u|^q\right)^{1/q}
 & \le C \{ 1+ r \sup_{\b(x_0, r)} m (\cdot, \B) \}^2 
\left(\fint_{\b(y, 2t)} | (D+\A) u|^p \right)^{1/p}\\
 &\qquad
 + C \{ 1+ r \sup_{\b(x_0, r)} m(\cdot, \B) \}^3 \frac{1}{r} \fint_{\b(x_0, r/2)}  |u|, 
 \endaligned
$$
if $\b(y, 2t) \subset \b(x_0, r/4)$.
It then follows by an iteration argument that
$$
\aligned
|(D+\A) u (x_0)|
 & \le C \{ 1+ r \sup_{\b(x_0, r)}  m(\cdot, \B)\}^N  \left(\fint_{\b(x_0, r/4)} |(D+\A) u |^2 \right)^{1/2}\\
&\qquad\qquad
+  C \{ 1+ r \sup_{\b(x_0, r)} m(\cdot, \B) \}^{N} \frac{1}{r} \fint_{\b(x_0, r/2)}  |u|\\
&\le C  \{ 1+ r m (x_0, \B) \}^{N_1} \frac{1}{r} \fint_{\b(x_0, r/2)} |u|, 
\endaligned
$$
for some $N, N_1\ge 1$, where we have used \eqref{Ca-0} and \eqref{m-2a} for the last step.
The desired estimate now follows readily from \eqref{I-2-1} by choosing $\ell = N_1$.
\end{proof}

\begin{thm}\label{g-thm-1}
Under the same assumptions as in Theorem \ref{I-thm-1}, we have
\begin{equation}\label{g-1}
|(D+\A) u (x_0)|
\le C \fint_{\b(x_0, r)} |(D+\A) u| \, dx 
+ C \fint_{\b(x_0, r)}  m(x, \B) |u(x)|\, dx,
\end{equation}
where $C$ depends on $C_0$ in \eqref{cod-1}.
\end{thm}

\begin{proof}

We consider two cases.

\noindent{\bf Case 1. } Assume that 
\begin{equation}\label{g-1a}
r^2 \sup_{\b(x_0, r)} |\B|\le   1.
\end{equation}
It follows from \eqref{I-3-1} that if $\b(y, 2t)\subset \b(x_0, r/4)$,
\begin{equation}\label{g-2}
\aligned
\left(\fint_{\b(y, t)} |(D+\A) u|^q \right)^{1/q}
 & \le
  C \left( 1+  t^2 \sup_{\b(y, 2t)} |\B|\right) 
 \left(\fint_{\b(y, 2t)} |(D+\A) u |^p \right)^{1/p}\\
 & \qquad\qquad
 + C t^2 \sup_{\b(y, 2t)} |\nabla \B|
 \fint_{\b(y, 3t)} |u|\\
 &\le C   \left(\fint_{\b(y, 2t)} |(D+\A) u |^p \right)^{1/p}
 + C r \sup_{\b(x_0, r)} |\B| \fint_{\b(x_0, r)} |u|,
 \endaligned
\end{equation}
where $1\le p< q\le \infty$,
$\frac{1}{p}-\frac{1}{q}< \frac{2}{d}$, and
 the fact $\sup_{\b(y, 2t)} |\nabla \B|\le C t^{-1} \sup_{\b(y, 2t)} |\B|$  
 is used for the
last step.
By iterating the estimate \eqref{g-2}, we obtain 
\begin{equation}\label{g-2a}
|(D+\A) u(x_0)|
\le C \fint_{\b(x_0, r)}  |(D+\A) u|
+ C r \sup_{\b(x_0, r)}  |\B| \fint_{\b(x_0, r)} |u|.
\end{equation}
Note that under the assumption \eqref{g-1a},
we have  $r< C m(x_0, \B)^{-1}$ and 
$$
\aligned
r \sup_{\b(x_0, r)} |\B| & \le C r \sup_{\b(x_0, r)} m(\cdot, \B)^2\\
 & \le C r m(x_0, \B)^2\\
 & \le C m (x_0, \B)
\le C  m(y, \B),
\endaligned
$$
for any $y\in \b(x_0, r)$.
This, together with \eqref{g-2a}, yields \eqref{g-1}.

\medskip

\noindent{\bf Case 2.} 
Assume that 
\begin{equation}\label{g-1b}
r^2 \sup_{\b(x_0, r)} |\B|>    1.
\end{equation}
Then $r\ge c m(x_0, \B)^{-1}$. It follows from \eqref{I-3-0} that
$$
\aligned
|(D+\A) u (x_0)|
 & \le C m(x_0, \B) \fint_{\b(x_0, r/2)} |u|\\
 &\le \frac{C_\ell m (x_0, \B)}{ \{ 1+ r m(x_0, \B) \}^\ell}
 \int_{\b(x_0, r)} |u|
\endaligned
$$
for any $\ell \ge 1$, where we have used \eqref{I-2-1} for the last inequality.
By \eqref{I-2-2}, we  obtain 
$$
|(D+\A) u (x_0)|
\le C \fint_{\b(x_0, r)}
m (x, \B) |u(x)|\, dx,
$$
which completes the proof.
\end{proof}

We introduce a modified nontangential maximal function.
For $u\in L^1_{loc} (\Omega; \C)$, define
\begin{equation}\label{m-max}
\mathcal{M}(u) (x)
=\sup \left\{  \fint_{\b(y, \delta(y)/2)} |u|: \ y\in \Omega \text{ and }  |y-x|<  M_0 \, \text{dist}(y, \partial\Omega) \right\}
\end{equation}
for $x\in \partial\Omega$.

\begin{cor}\label{cor-M}
Assume $\B$ satisfies the same conditions as in Theorem \ref{I-thm-1}.
Suppose  $u\in C^2(\overline{\Omega}; \C)$ and $(D+\A)^2 u =0$ in $\Omega$. 
Let $v(x) =|(D+\A) u(x)|  + m(x, \B) |u(x)|$.
Then
\begin{equation}\label{I-0}
v(y)\le C \fint_{\b(y, r)} v
\end{equation}
for any $\b(y, r) \subset \Omega$. Consequently, 
\begin{equation}\label{mm-1}
(v)^* (x) \le C \mathcal{M} (v) (x)
\end{equation}
for any $x\in \partial\Omega$, where $C$ depends on the Lipschitz character of $\Omega$ and $C_0$ in \eqref{cod-1}.
\end{cor}

\begin{proof}

This follows readily from \eqref{I-2-0} and \eqref{g-1}.
\end{proof}


\section{The Green function}\label{sec-G}

In this section we establish decay estimates for the Green function $G_\A(x, y)$  in  a bounded Lipschitz domain $\Omega$.

\begin{lemma}\label{lemma-G1}
Let  $x_0\in \partial\Omega$ and  $0< r< r_0$.
 Let $u\in C^2(\b(x_0, r) \cap \Omega; \C)\cap C (\b(x_0, r) \cap \overline{\Omega}; \C)$ be a  solution of
\begin{equation}\label{G1-1}
\left\{
\aligned
(D +\A)^2 u  &  =0   & \quad &   \text{ in }\  \b(x_0, r)\cap \Omega, \\
u& =0 &  \quad & \text{ on }\  \b(x_0, r) \cap \partial\Omega.
\endaligned
\right.
\end{equation}
Then
\begin{equation}\label{G1-2}
\sup_{\b(x_0, r/2) \cap \Omega}  |u|\le C  \fint_{\b(x_0, r)\cap \Omega} |u|,
\end{equation}
where $C$ depends  on the Lipschitz character of $\Omega$.
\end{lemma}

\begin{proof}

Let $\e>0$ and $v_\e =\sqrt{ |u|^2 +\e^2}-\e\ge 0 $.
It follows from \eqref{p1-2}  and  \eqref{G1-1} that  $\Delta v_\e \ge 0$. 
 Thus,  $v_\e$ is subharmonic in $\b(x_0, r) \cap \Omega$ and
$v_\e = 0 $ on $\b(x_0, r)\cap \partial \Omega$.
Hence, 
$$
\sup_{\b(x_0, r/2) \cap \Omega}  v_\e \le C \fint_{\b(x_0, r)\cap \Omega}  v_\e,
$$
where $C$ depends only on $\Omega$.
By letting $\e \to 0$, we obtain \eqref{G1-2}.
\end{proof}

Consider the Dirichlet problem, 
\begin{equation}\label{DP-1}
\left\{
\aligned
(D+\A)^2 u & =F  & \quad & \text{ in } \ \Omega,\\
u &=0 & \quad & \text{ on } \ \partial\Omega.
\endaligned
\right.
\end{equation}

\begin{lemma}\label{lemma-G2}
For $F\in H^{-1}(\Omega; \C)$,
 the Dirichlet problem \eqref{DP-1} has a unique weak solution in $H_0^1(\Omega; \C)$.
Moreover, if $F\in L^2(\Omega; \C)$,
 the solution satisfies 
\begin{equation}\label{G3-0}
\| u \|_{L^{p_d}(\Omega)} + \| (D +\A) u \|_{L^2(\Omega)} \le C \| F \|_{L^{p_d^\prime}(\Omega)},
\end{equation}
where $p_d=\frac{2d}{d-2}$ for $d\ge 3$ and $2\le p_d< \infty$ for $d=2$. The constant $C$ in \eqref{G3-0}
depends at most on $d$, $p_d$ and $\Omega$.
\end{lemma}

\begin{proof}

Consider the bilinear  form
\begin{equation}\label{bi}
B[u, v]=\int_\Omega (D+\A) u \cdot \overline{(D+\A) v}
\end{equation}
for $u, v\in H_0^1(\Omega; \C)$.
Since $\A$ is bounded in $\Omega$, we have
$$
\aligned
B [u, u]  & \ge c\int_\Omega |\nabla u|^2 - C(\A) \int_\Omega |u|^2\\
&\ge c \int_\Omega |\nabla u|^2- C(\A) B[u, u],
\endaligned
$$
where  $c>0$, $C (\A)$ depends on $\A$,  and we have used \eqref{P-1} for the second inequality.
It follows that 
$$
B[u, u] \ge \frac{c}{1+ C(\A)}
\int_\Omega |\nabla u|^2.
$$
Thus, by the Lax-Milgram Theorem,  \eqref{DP-1} has a unique weak solution in $H_0^1(\Omega; \C)$
for any $F\in H^{-1}(\Omega; \C)$. 
To see \eqref{G3-0}, we note that
\begin{equation}\label{G3-1}
\int_\Omega |(D+\A)u|^2 \le  \int_\Omega | u| |F|
\le \| u \|_{L^{p_d}(\Omega)} \|F \|_{L^{p_d^\prime} (\Omega)}.
\end{equation}
This,  together with \eqref{P-1}, gives \eqref{G3-0}.
\end{proof}

\begin{remark}
{\rm  Suppose $F\in L^\infty (\Omega; \C)$.
Since
\begin{equation}\label{G4-0}
(D+\A)^2 u =-\Delta u + 2 \A\cdot D u +(D\cdot \A +|\A|^2) u,
\end{equation}
under the assumption $\A \in C^1(\overline{\Omega}; \R^d)$ and $\Omega$ is Lipschitz, the weak solution of \eqref{DP-1}  is H\"older continuous in $\overline{\Omega}$.
}
\end{remark}

\begin{remark}
{\rm

Let $f\in H^{1/2}(\partial\Omega; \C)$. Choose $\Phi\in H^1(\Omega; \C)$ such that
$\Phi=f$ on $\partial\Omega$
and 
$\| \Phi \|_{H^1(\Omega)}\le C \| f\|_{H^{1/2}(\partial\Omega)}$.
By considering $u-\Phi$ and applying Lemma \ref{lemma-G2}, one obtains a unique week solution in $H^1(\Omega; \C)$ for the Dirichlet problem
\eqref{DP0}.
}
\end{remark}

\begin{thm}\label{thm-G}
There is a continuous  function $G_\A: \{ (x, y)\in \overline{\Omega} \times \overline{\Omega}: x\neq y \} \to \C$ with the following properties.

\begin{enumerate}

\item
 For any $x, y\in \Omega$, $x\neq y$,  and $\sigma \in (0, 1)$, 
\begin{equation}\label{G5-0}
|G_\A (x, y)|  \le 
\left\{
\aligned
 & C |x-y|^{2-d} & \quad & \text{ if } d\ge 3,\\
 & C_\sigma |x-y|^{-\sigma} & \quad & \text{ if } d=2,
 \endaligned
 \right.
\end{equation}
where  $C$ depends only on $r_0$ and  the Lipschitz character of $\Omega$, and $C_\sigma$ depends only on $\sigma$ and $\Omega$. 

\item

For any $x, y \in \Omega$, $x \neq y$, 

\begin{equation}\label{G5-1}
G_\A (x, y) =\overline{G_\A (y, x)}.
\end{equation}

\item

Fix $y \in \Omega$. Then  $G_\A (\cdot, y) \in W^{1, 2}_{\loc} (\Omega \setminus \{ y\}; \C)$,  $G_{\A} (\cdot, y) =0$ on $\partial\Omega$, and
\begin{equation}\label{G5-2}
(D+\A)^2 G_\A (\cdot, y) =0 \quad \text{ in } \Omega \setminus \{ y \}.
\end{equation}

\item

For  $F\in L^\infty (\Omega; \C)$, the weak solution of \eqref{DP-1}  is given by 
\begin{equation}\label{G5-3}
u(x)=\int_\Omega G_\A(x, y) F (y)\, dy.
\end{equation}

\end{enumerate}
\end{thm}

\begin{proof}

The proof is similar to that in the case of second-order elliptic operators in divergence form \cite{GW-1982, HK-2007}.
Fix $ y\in \Omega$ and $\rho  \in (0, 1)$.
Let $G_\A^\rho (\cdot, y)$ denote the weak solution of \eqref{DP-1} with $F=\frac{1}{|\b(y, \rho)\cap \Omega|} \chi_{\b(y, \rho)\cap \Omega}$, 
given by Lemma \ref{lemma-G2}. Thus,
\begin{equation}\label{G5-6}
\int_\Omega (D+\A) G_\A^\rho (\cdot, y) \cdot \overline{(D+\A) \psi}
= \fint_{\b(y, \rho)\cap \Omega} \overline{\psi}
\end{equation}
for any $\psi \in H_0^1(\Omega; \C)$.
 It follows from \eqref{G3-0} that
\begin{equation}\label{G5-6a}
\| G_\A^\rho (\cdot, y) \|_{L^{p_d}(\Omega)} 
+\| (D+\A) G^\rho_\A (\cdot, y) \|_{L^2(\Omega)}
\le C \rho^{-\frac{d}{p_d}}, 
\end{equation}
where $p_d=\frac{2d}{d-2}$ for $d\ge 3$ and $2\le p_d< \infty$ for $d=2$.

Next, we fix $z\in \Omega$ and $z\neq y$.
Let $r= (1/2) |y-z|$.
For $F \in  C_0^\infty(\b(z, r)\cap \Omega; \C)$, let $u\in H_0^1(\Omega; \C)$ be the weak solution of \eqref{DP-1}.
Since $(D+\A)^2 u=0$ in $\Omega \setminus  \b(z, r)$ and $u=0$ on  $\partial\Omega$, it follows by 
\eqref{G1-2}, \eqref{G2} and \eqref{G3-0}  that
\begin{equation}\label{G5-7}
\aligned
\sup_{\b (y, r/2)\cap \Omega} |u|
 & \le C \left(\fint_{\b(y, r)\cap \Omega} |u|^{p_d} \right)^{1/p_d}\\
 & \le C r^{-\frac{d}{p_d}} \| u \|_{L^{p_d}(\Omega)}
\le C r^{-\frac{d}{p_d}} \| F \|_{L^{{p_d ^\prime}} (\Omega)}.
\endaligned
\end{equation}
Also, note that 
\begin{equation}\label{G5-8a}
\int_\Omega (D+\A) u \cdot \overline{(D+\A) \psi} =\int_\Omega F \cdot \overline{\psi}
\end{equation}
for any $\psi \in H_0^1(\Omega; \C)$.
In view of  \eqref{G5-6} and \eqref{G5-8a}, we have 
\begin{equation}\label{G5-9a}
\int_\Omega F(x)  \cdot \overline{G_\A^\rho (x, y)}\, dx 
=\fint_{\b(y, \rho)\cap \Omega} u.
\end{equation}
This, together with \eqref{G5-7}, gives
$$
\Big |\int_\Omega F(x)  \cdot \overline{G_\A^\rho (x, y)}\, dx \Big |
\le C r^{-\frac{d}{p_d}} \| F \|_{L^{p_d^\prime}(\Omega)}, 
$$
provided $0< \rho< (1/2) r$. By duality, we obtain 
\begin{equation}\label{G5-8}
\| G_\A^\rho (\cdot,  y)\|_{L^{p_d}(\b(z, r) \cap \Omega)}\le C r^{-\frac{d}{p_d}}.
\end{equation}
Since $G_\A^\rho (\cdot, y) \in H^1_0(\Omega; \C)$ and
\begin{equation}
(D+\A)^2 G_\A^\rho (\cdot, y) =0 \quad \text{ in } \Omega \setminus \b(y, \rho),
\end{equation}
it follows from Lemma \ref{lemma-G1}  and \eqref{G5-8} that  if $\rho< (1/2) r$,
\begin{equation}
\aligned
|G_\A^\rho (z, y)|
 & \le C \left(\fint_{\b(z, r )\cap \Omega} |G_\A^\rho(x, y)|^{p_d} dx  \right)^{1/p_d}\\
 & \le C r^{-\frac{2d}{p_d}}.
\endaligned
\end{equation}
Recall that $p_d=\frac{2d}{d-2}$ for $d\ge 3$ and $2\le p_d < \infty$ for $d=2$.
Thus,  we have proved that  if $\rho<  (1/4) |z-y|$, 
\begin{equation}\label{G5-9}
|G_\A^\rho (z, y)|\le 
\left\{
\aligned
 &C |z-y|^{2-d} & \quad & \text{ if } d\ge 3,\\
 & C_\sigma |z-y|^{-\sigma}  & \quad & \text{ if } d=2,
 \endaligned
 \right.
 \end{equation}
  for any $\sigma \in (0, 1)$.
  
  Using the estimates \eqref{G5-9} and \eqref{G5-6a}, one can show that $\{ G^\rho_\A (\cdot, y):  0< \rho< 1 \}$ is  bounded
  in $L^p(\Omega)$ for $1< p< \frac{d}{d-2}$ if $d\ge 3$,  and in $L^p(\Omega)$ for any $p< \infty$ if $d=2$. 
  One can also show that $\{ (D+\A) G^\rho_\A (\cdot, y): 0< \rho< 1\}$ is  bounded in $L^p(\Omega)$ for $p< \frac{d}{d-1}$.
  It follows that $\{ G_\A^\rho (\cdot, y): 0< \rho< 1 \}$ is  bounded in $W^{1, p}(\Omega; \C)$ for $1< p< \frac{d}{d-1}$ and
  $d\ge 2$. This implies that there exists  a subsequence  $\{ G^{\rho_j}_\A(\cdot, y) \}$, where $\rho_j \to 0$,  and $G_\A(\cdot, y) \in
  W^{1, p} (\Omega, \C)$ such that
  $$
  G^{\rho_j} _\A (\cdot, y) \to G_\A (\cdot , y)  \quad \text{ weakly in } W^{1, p}(\Omega; \C),
  $$
  where $1< p< \frac{d}{d-1}$.
  By Lemma \ref{lemma-Ca}, $\{ G_\A^\rho (\cdot, y): 0< \rho< 1\}$ is bounded in $W^{1, 2}(\Omega\setminus \b(y, t); \C)$ for any  fixed $t\in (0, 1)$.
  It follows that $G_\A (\cdot, y)\in W^{1, 2}_{loc} (\Omega\setminus \{ y\}; \C)$ and \eqref{G5-2} holds.
 
 Since $\A\in C^1(\overline{\Omega}; \R^d)$, using \eqref{G5-9} and  the standard elliptic regularity theory, one may deduce that  
 $\{ G_\A^\rho (\cdot, y): 0< \rho< 1 \}$ is equicontinuous  in $\Omega \setminus \b(y, t)$ for any fixed $t\in (0, 1)$.
 Thus, we may assume $G_\A^{\rho_j }  (\cdot, y) \to G_\A (\cdot, y) $ uniformly in $\Omega\setminus \b(y, t)$
 for any fixed $t\in (0, 1)$. As a result,   \eqref{G5-0} follows from \eqref{G5-9}.
 To show \eqref{G5-1}, we use \eqref{G5-9a} with  $F= \frac{1}{|\b(z, \rho)\cap \Omega|}\chi_{\b(z, \rho)\cap\Omega} $ and
$u=G_\A^\rho (\cdot , z) $ to obtain 
$$
\fint_{B(z, \rho)\cap \Omega} \overline{G_\A^\rho (x, y)}\, dx
=\int_{\b(y, \rho)\cap \Omega} G_\A^\rho (x, z)\, dx.
$$
By letting $\rho=\rho_j \to 0$, we arrive at  \eqref{G5-1}.

 Finally, let $u$ be the weak solution of \eqref{DP-1} with $F\in L^\infty(\Omega; \C)$.
 Since $\A\in C^1(\overline{\Omega}; \R^d)$, $u$ is H\"older continuous in $\overline{\Omega}$.
By taking limits in \eqref{G5-9a}, we obtain 
\begin{equation}
u(y)=\int_\Omega F(x) \cdot \overline{G_\A (x, y)} \, dx,
\end{equation}
which, together with \eqref{G5-1}, gives \eqref{G5-3}.
\end{proof}

\begin{thm}\label{b-thm-1}
Let $\Omega\subset \b(0, R_0)$ be a bounded Lipschitz domain.
Suppose $\B$ satisfies the condition \eqref{cod-1} for any ball $\b(x, r) \subset \b(0, 4R_0)$.
Also assume that \eqref{cod-1a} holds for any $x\in \b(0, R_0)$.
Let $u$ be a solution of \eqref{G1-1}, where $x_0\in \partial\Omega$ and $0< r< r_0$.
Then for any $\ell\ge 0$, 
\begin{equation}\label{b-1-0}
\sup_{\b(x_0, r/2)\cap \Omega} |u|
\le \frac{C_\ell}{ \{ 1+ r m(x_0, \B )\}^\ell} \fint_{\b(x_0, r)\cap \Omega} |u|, 
\end{equation}
where $C_\ell$ depends on $\ell$, the Lipschitz character of $\Omega$ and $C_0$ in \eqref{cod-1}.
\end{thm}

\begin{proof}

This follows readily from Lemma \ref{lemma-G1} and Theorem \ref{thm-Ca-m}.
\end{proof}

\begin{thm}\label{thm-G1}
Assume $\Omega$ and $\B$ satisfy  the same conditions as in Theorem \ref{b-thm-1}.
Let $G_\A(x, y) $ be the Green function for the operator  $(D+\A)^2$ in $\Omega$, given by Theorem \ref{thm-G}.
Let $\ell \ge 1$.
Then,  for any $x, y\in \Omega$ and $x\neq y$,
\begin{equation}\label{mg1}
|G_\A(x, y) |
\le \frac{C_\ell}{ \{ 1+ |x-y| m (x, \B)  \}^\ell } \cdot \frac{1}{|x-y|^{d-2}}
\end{equation}
for $d\ge 3$, where $C_\ell $ depends on $\ell$, $r_0$, the Lipschitz character of $\Omega$ and $C_0$ in \eqref{cod-1}. If $d=2$, we have
\begin{equation}\label{mg2}
|G_\A(x, y) |
\le \frac{C_{\ell, \sigma} }{ \{ 1+ |x-y| m (x, \B)  \}^\ell } \cdot \frac{1}{|x-y|^\sigma}
\end{equation}
for any $\sigma \in (0, 1)$, where $C_{\ell, \sigma}$ depends on $\ell$, $\sigma$, $r_0$, the Lipschitz character of $\Omega$ and $C_0$ in \eqref{cod-1}.
\end{thm}

\begin{proof}

We gives the proof for  $d\ge 3$.
The proof for  $d=2$ is the same.
Fix $x_0, y_0 \in \Omega$ and $x_0 \neq y_0$. Let $u(x) = G_\A (x, y_0)$ and $r=(1/8)\min \{   |x_0-y_0|, r_0\}$. We consider two cases.

Case 1. Suppose dist$(x_0, \partial\Omega)>  r$. Then $\b(x_0, r)\subset \Omega$.
It follows from \eqref{I-2-00} that
$$
\aligned
|u(x_0)| & \le \frac{C_\ell}{ \{ 1+ r m(x_0, \B) \}^\ell} \fint_{\b(x_0, r)} |u| \\
& \le \frac{C_\ell}{ \{ 1+ r m(x_0, \B) \}^\ell} \cdot \frac{1}{r^{d-2}}, 
\endaligned
$$
where we have used \eqref{G5-0} for the last inequality. This gives \eqref{mg1}

Case 2. Suppose dist$(x_0, \partial\Omega)\le r$. Choose $z_0\in \partial\Omega$ such that
$|x_0-z_0|=\text{\rm dist}(x_0, \partial\Omega)$.
By \eqref{b-1-0} we have
\begin{equation}\label{mg3}
\aligned
 |u(x_0)| & \le \sup_{\b(z_0, 2r)} |u|
  \le \frac{C_\ell}{ \{ 1+ r m(z_0, \B) \}^\ell} \fint_{\b(z_0, 4r)\cap \Omega} |u|\\
 & \le \frac{C_\ell}{ \{ 1+ r m(z_0, \B) \}^\ell} \cdot \frac{1}{r^{d-2}}, 
\endaligned
\end{equation}
where we have used \eqref{G5-0} for the last inequality. 
Finally, note that by \eqref{m-2a},
$$
m(x_0, \B) \le C \{ 1+ r m (z_0, \B) \}^{\kappa_0} m(z_0, \B).
$$
It follows that
$$
1+ r m (x_0, \B)\le C \{ 1+ r m(z_0, \B)\}^{\kappa_0+1}.
$$
Since $\ell\ge 1$ is arbitrary, this, together with \eqref{mg3}, gives \eqref{mg1}.
\end{proof}

\begin{thm}\label{thm-G2}
Let $d\ge 3$.
Assume $\Omega$ and $\B$ satisfy  the same conditions as in Theorem \ref{b-thm-1}.
For $F\in L^\infty(\Omega; \C)$, let $u$ be a weak solution of \eqref{DP-1}. Then
\begin{equation}\label{G2-0}
\| m(x, \B)^{\ell +2} u \|_{L^p(\Omega)}
\le C_\ell \| m(x, \B)^\ell  F \|_{L^p(\Omega)}
\end{equation}
for any $\ell \in \R$ and $1\le p\le \infty$, where $C_\ell $ depends on  $\ell$,  $r_0$,
the Lipschitz character of $\Omega$, and $C_0$ in \eqref{cod-1}.
\end{thm}

\begin{proof}

By Theorem \ref{thm-G}, we may write
$$
m(x, \B)^{\ell+2} u(x)   =\int_\Omega K_\ell (x, y) m(y, \B)^\ell F(y)\, dy,
$$
where
$$
K_\ell (x, y) =m(x, \B)^{\ell +2} G_\A(x, y) m(y, \B)^{-\ell}.
$$
To show \eqref{G2-0}, it suffices to show that  for any $x\in \Omega$,
\begin{equation}\label{G2-1}
\int_\Omega |K_\ell (x, y)|\, dy \le C_\ell, 
\end{equation}
and for any $y\in \Omega$,
\begin{equation}\label{G2-2}
\int_\Omega |K_\ell (x, y)|\, dx\le C_\ell.
\end{equation}
To see \eqref{G2-1}, we use the estimate
\begin{equation}
m(x, \B) \le C \{ 1+|x-y| m(y, \B) \}^{\kappa_0} m (y, \B)
\end{equation}
for any $x, y \in \Omega$ and the fact  that \eqref{mg1} holds for any $\ell\ge 1$.
This gives
$$
|K_\ell  (x, y)| \le \frac{C_{\ell}  \, m (x, \B)^2 }{ \{ 1+|x-y| m (x, \B)\}^\ell |x-y|^{d-2}} 
$$
for any $\ell\ge 1$, from which the inequality \eqref{G2-1} follows readily.
The proof for \eqref{G2-2} is similar.
\end{proof}


\section{The Neumann function}\label{sec-N}

In this section we establish decay estimates  for  the Neumann function $N_{\A} (x, y)$
in  a bounded Lipschitz domain $\Omega$.

\begin{lemma}\label{lemma-N1}
Let  $x_0\in \partial\Omega$ and  $0< r< r_0$.
 Let $u\in C^2(\b(x_0, r) \cap \Omega; \C)\cap C^1 (\b(x_0, r)\cap \overline{\Omega}; \C)$ be a  solution of
\begin{equation}\label{N1-0}
\left\{
\aligned
(D +\A)^2 u  &  =0   & \quad &   \text{ in }\  \b(x_0, r)\cap \Omega, \\
n \cdot (D+\A) u & =0 &  \quad & \text{ on }\ \b(x_0, r) \cap \partial\Omega,
\endaligned
\right.
\end{equation}
where $n$ denotes the outward unit normal to $\partial\Omega$.
Then
\begin{equation}\label{N1-1}
\sup_{\b(x_0, r/2) \cap \Omega}  |u|\le C   \fint_{\b(x_0, r)\cap \Omega} |u|,
\end{equation}
where $C$ depends only on $\Omega$.
\end{lemma}

\begin{proof}

Let  $\e>0$ and $u_\e =\sqrt{|u|^2 +\e^2}$.
Since $(D+\A)^2 u=0$ in $\b(x_0, r) \cap \Omega$, by \eqref{p1-2},
we have  $\Delta u_\e \ge 0$ in $\b(x_0, r)\cap \Omega$.
Using $n \cdot (D+\A) u=0$ on $\b(x_0, r) \cap \partial\Omega$ and \eqref{p1-3},
we see that  $n \cdot \nabla u_\e =0$ on $\b(x_0, r) \cap \partial\Omega$.
By the even reflection, one can show that $u_\e$ is a nonnegative sub-solution of
a second-order elliptic equation in divergence form with bounded measurable coefficients in
$\b(x_0, cr)$. 
More precisely, without the loss of generality, assume that 
$$
\b(x_0, r_0) \cap \Omega
=\b(x_0, r_0)\cap \{ (x^\prime, x_d) \in \R^d: x_d> \phi (x^\prime) \}, 
$$
where $\phi: \R^{d-1} \to \R$ is a Lipschitz continuous function with $\phi (x_0)=0$ and $\|\nabla \phi\|_\infty
\le M_0$. Let $\Phi$ be the map
$$
(x^\prime, x_d) \to (x^\prime, 2\phi(x^\prime)-x_d)
$$
from $\R^d$ to $\R^d$.
Note that $\Phi^{-1} = \Phi$.
For $x\in \b(x_0, cr)$, define
$$
v_\e (x)
=\left\{
\aligned   & u_\e(x) & \quad & \text{ if } x_d\ge  \phi (x^\prime),\\
& u_\e(\Phi(x)) & \quad & \text{ if } x_d< \phi (x^\prime).
\endaligned
\right.
$$
Then $v_\e \in H^1(\b(x_0, cr); \R)$ and 
$$
\int_{\b(x_0, cr)} E  \nabla v_\e \cdot \nabla \varphi\le 0
$$
for any nonnegative $\varphi \in C_0^1(\b(x_0, cr); \R)$, where
$$
E=\left\{
\aligned
& I & \quad & \text{ if } x_d\ge \phi(x^\prime),\\
& \left( \frac{\partial \Phi}{\partial x}\right) \left(\frac{\partial \Phi}{\partial x} \right)^T (\Phi^{-1}(x))
& \quad & \text{ if } x_d< \phi(x^\prime).
\endaligned
\right.
$$
Note that $E$ is symmetric.
The Lipschitz condition $\|\nabla \phi\|_\infty\le M_0$ ensures that $E$ is bounded measurable and uniformly elliptic.
As a result, by the De Giorgi-Nash-Moser  estimates, we  obtain 
$$
\sup_{B(x_0, cr/2)\cap \Omega} u_\e
\le C \fint_{B(x_0, cr ) \cap \Omega} u_\e,
$$
where $C$ depends only on the Lipschitz character of $\Omega$. By letting $\e \to 0$, this leads to 
$$
\sup_{B(x_0, cr/2)\cap \Omega} |u|
\le C \fint_{B(x_0, cr ) \cap \Omega} |u|,
$$
from which and the interior estimate \eqref{I-1},   the estimate \eqref{N1-1} follows by a simple covering argument.
\end{proof}

In the remaining of this section we will assume that  there exists $c_0>0$ such that 
\begin{equation}\label{H-1}
c_0 \int_\Omega | \psi  |^2
\le \int_\Omega 
|(D+\A) \psi  |^2
\end{equation}
for any $\psi  \in H^1(\Omega;  \C)$.
It follows from \eqref{H-1} and  \eqref{P-1a} that 
\begin{equation}\label{P-2}
\left(\int_\Omega | \psi  |^{p_d} \right)^{1/p_d}
\le C \left(\int_\Omega 
|(D+\A) \psi  |^2 \right)^{1/2}
\end{equation}
for any $ \psi  \in H^1(\Omega; \C)$, where
$p_d=\frac{2d}{d-2}$ for $d\ge 3$ and $2\le  p_d< \infty$ for $d=2$.
The constant $C$ in \eqref{P-2} depends on 
 $p_d$, $\Omega$ and $c_0$ in \eqref{H-1}.

Consider the Neumann  problem,
\begin{equation}\label{NP-1}
\left\{
\aligned
(D+\A)^2 u & =F  & \quad & \text{ in } \ \Omega,\\
n\cdot (D+\A) u &=g & \quad & \text{ on } \ \partial\Omega.
\endaligned
\right.
\end{equation}
Let $F\in L^2(\Omega; \C)$ and $g\in L^2(\partial\Omega; \C)$, 
 we say $u\in H^1(\Omega; \C)$ is a weak solution of
\eqref{NP-1} if
\begin{equation}\label{weak-N}
\int_\Omega (D+\A)u \cdot \overline{ (D+\A) \psi } =\int_\Omega F \cdot \overline{\psi}
+i \int_{\partial\Omega} g \cdot \overline{\psi}
\end{equation}
for any $\psi \in H^1(\Omega; \C)$.

\begin{lemma}\label{lemma-N2}
For $F\in L^2(\Omega; \C)$ and $g\in L^2(\partial\Omega; \C)$, the Neumann problem \eqref{NP-1} possesses  a unique weak solution in $H^1(\Omega; \C)$.
Moreover,  the solution satisfies 
\begin{equation}\label{local-30n}
\| u \|_{L^{p_d}(\Omega)} + \| (D +\A) u \|_{L^2(\Omega)} \le C \left\{  \| F \|_{L^{p_d^\prime}(\Omega)} + \| g \|_{L^2(\partial\Omega)} \right\},
\end{equation}
where $p_d=\frac{2d}{d-2}$ for $d\ge 3$ and $2\le p_d< \infty$ for $d=2$. The constant $C$ in \eqref{local-30n}
depends at most on $p_d$, $\Omega$ and $c_0$ in \eqref{H-1}.
\end{lemma}

\begin{proof}

Consider the bilinear  form $B[u, v]$ in \eqref{bi} 
for $u, v\in H^1(\Omega; \C)$.
Since $\A$ is bounded in $\Omega$, we have
$$
\aligned
B [u, u]  & \ge \frac12 \int_\Omega |\nabla u|^2 - C(\A) \int_\Omega |u|^2\\
&\ge \frac12  \int_\Omega |\nabla u|^2- C(\A) B[u, u],
\endaligned
$$
where   $C (\A)$ depends on $\A$ and we have used the assumption  \eqref{H-1} for the second inequality.
It follows that 
$$
B[u, u] \ge  c(\A)  \| u \|^2_{H^1(\Omega)}
$$
for any $u\in H^1(\Omega; \C)$, where $c(\A)>0$. By the Lax-Milgram Theorem and \eqref{P-1c}, this implies that \eqref{NP-1} has a unique weak solution in $H^1(\Omega; \C)$ for any $F\in L^2(\Omega; \C)$ and $g\in L^2(\partial\Omega; \C)$. 
To see \eqref{local-30n}, we use \eqref{weak-N} with $\psi=u$,  \eqref{P-2} and \eqref{P-1c}.
\end{proof}

\begin{thm}\label{thm-N}
There is a continuous function $N_\A: \{ (x, y)\in \overline{\Omega} \times \overline{\Omega}: x\neq y \} \to \C$ with the following properties.

\begin{enumerate}
 
\item

 For any $x, y\in \overline{\Omega}$, $x\neq y$,  and $\sigma \in (0, 1)$, 
\begin{equation}\label{N-0}
|N_\A (x, y)|  \le 
\left\{
\aligned
 & C |x-y|^{2-d} & \quad & \text{ if } d\ge 3,\\
 & C_\sigma |x-y|^{-\sigma} & \quad & \text{ if } d=2,
 \endaligned
 \right.
\end{equation}
where  $C$ depends  on $d$, $\Omega$ and $c_0$ in \eqref{H-1}, and $C_\sigma$ depends also on $\sigma$. 

\item

For any $x, y \in \overline{\Omega}$, $x \neq y$, 

\begin{equation}\label{N-1}
N_\A (x, y) =\overline{N_\A (y, x)}.
\end{equation}

\item

Fix $y \in \overline{\Omega}$. Then $N_\A (\cdot, y) \in W^{1, 2}_{\loc} (\Omega \setminus \{ y\}; \C)$ and
\begin{equation}\label{N-0a}
\left\{
\aligned
(D+\A)^2 N_\A (\cdot, y)  & =0  & \quad &  \text{ in }\  \Omega \setminus \{ y \}, \\
n\cdot (D+\A) N_\A (\cdot, y)  & =0 & \quad & \text{ on }\  \partial \Omega.
\endaligned
\right.
\end{equation}

\item

For  $F\in L^\infty (\Omega; \C)$ and $g\in L^\infty(\partial\Omega)$, the weak solution of \eqref{NP-1}  is given by 
\begin{equation}\label{N11}
u(x)=\int_\Omega N_\A(x, y) F (y)\, dy + i \int_{\partial\Omega} N_\A(x, y) g(y) dy.
\end{equation}

\end{enumerate}
\end{thm}

\begin{proof}

With Lemmas \ref{lemma-N1} and \ref{lemma-N2} at our disposal,
the proof is similar to that of Theorem \ref{thm-G}. Indeed, 
fix $y\in \overline{\Omega}$ and $\rho \in (0, 1)$.
Let $N_\A^\rho (\cdot, y)$ denote the weak solution of \eqref{NP-1} with $F=\frac{1}{|\b(y, \rho)\cap \Omega|}
\chi_{\b(y, \rho)\cap \Omega}$ and $g=0$, given by Lemma \ref{lemma-N2}. 
Thus,
\begin{equation}\label{N-11a}
\int_\Omega (D+\A) N_\A^\rho (x, y) \cdot
\overline{(D+\A)\psi} \, dx
=\fint_{\b(y, \rho)\cap \Omega} \overline{\psi}
\end{equation}
for any $\psi \in H^1(\Omega; \C)$.
Using \eqref{local-30n} and \eqref{N1-1},
one can show that if $\rho< (1/4) |z-y|$,
$$
|N_\A(z, y)|\le 
\left\{
\aligned 
& C |z-y|^{2-d} & \quad & \text{ if } d\ge 3,\\
& C_\sigma |z-y|^{-\sigma} & \quad & \text{ if } d=2,
\endaligned
\right.
$$
for any $\sigma \in (0, 1)$. It follows that $\{ N_\A^\rho(\cdot, y): 0< \rho< 1 \}$
is bounded in $W^{1, p}(\Omega; \C)$ for $p< \frac{d}{d-1}$.
This implies that there exists a sequence $\{N_\A^{\rho_j} (\cdot, y) \}$,
where $\rho_j \to 0$,  and $N_\A^\rho (\cdot, y) \in W^{1, p}(\Omega; \C)$
such that $N_\A^{\rho_j} (\cdot, y) \to N_\A (\cdot, y)$ weakly in $W^{1, p}(\Omega; \C)$,
where $1< p< \frac{d}{d-1}$.
The proof for \eqref{N-0}-\eqref{N-0a} is the same as in the case of the Green function.
To see \eqref{N11}, let $u$ be a weak solution of \eqref{NP-1} with $F \in L^\infty(\Omega; \C)$ and
$g\in L^\infty(\partial\Omega; \C)$.
By letting $\psi=N_\A^{\rho} (\cdot, y)$ in \eqref{weak-N}, we obtain 
$$
\int_\Omega (D+\A) u(x)  \cdot \overline{(D+\A) N^\rho_\A(x, y)}\, dx
=\int_\Omega F(x) \overline{N^\rho_\A(x, y)}\, dx
+i \int_{\partial\Omega} g(x)  \overline{N^\rho_\A(x, y)}\, dx.
$$
This, together with \eqref{N-11a}, yields 
\begin{equation}\label{N11b}
\fint_{\b(y, \rho)\cap \Omega} u
=\int_\Omega F(x) \overline{N^\rho_\A(x, y)}\, dx
+i \int_{\partial\Omega} g(x)  \overline{N^\rho_\A(x, y)}\, dx.
\end{equation}
Under the assumptions that $F\in L^\infty(\Omega; \C)$ and $g\in L^\infty(\partial\Omega; \C)$,
using the standard regularity theory for $\Delta$ in a Lipschitz domain,
one may show that $u$ is H\"older continuous in $\overline{\Omega}$.
By letting $\rho =\rho_j\to 0$ in \eqref{N11b}, we see that
$$
u(y)
=\int_\Omega F(x) \overline{N_\A(x, y)}\, dx
+i \int_{\partial\Omega} g(x)  \overline{N_\A(x, y)}\, dx.
$$
In view of \eqref{N-1}, this gives the representation formula \eqref{N11}.
\end{proof}

\begin{thm}\label{thm-N1}
Let $\Omega\subset \b(0, R_0)$ be a bounded Lipschitz domain.
Suppose that $\B$ satisfies the condition \eqref{cod-1} for any ball $\b(x, r) \subset \b(0, 4R_0)$.
Also assume that \eqref{cod-1a} holds for any $x\in \b(0, R_0)$.
Let $N_\A(x, y)$ be the Neumann function for the operator  $(D+\A)^2$ in $\Omega$, given by Theorem \ref{thm-N}.
Let $\ell \ge 1$.
Then,  for any $x, y\in \Omega$ and $x\neq y$,
\begin{equation}\label{mN}
|N_\A(x, y) |
\le \frac{C_\ell }{ \{ 1+ |x-y| m (x, \B)  \}^\ell  } \cdot \frac{1}{|x-y|^{d-2}}
\end{equation}
for $d\ge 3$, where $C_\ell $ depends on $\ell$, $R_0$, $\Omega$, and $C_0$ in \eqref{cod-1}. If $d=2$, we have
\begin{equation}\label{mN1}
|N_\A(x, y) |
\le \frac{C_{\ell, \sigma} }{ \{ 1+ |x-y| m (x, \B)  \}^\ell } \cdot \frac{1}{|x-y|^\sigma}
\end{equation}
for any $\sigma \in (0, 1)$, where $C_{\ell, \sigma}$ depends on $\ell$, $\sigma$, $R_0$, the Lipschitz character of
$\Omega$ and $C_0$ in \eqref{cod-1}.
\end{thm}

\begin{proof}

The proof is similar to that of Theorem \ref{thm-G1}.
We only point out that under the condition \ref{cod-1a}, we have $m(x, \B)\ge 2R_0^{-1}$ for any $x\in \Omega$.
In view of Theorem \ref{thm-M1}, this leads to 
$$
cR_0^{-2} \int_\Omega |\psi|^2 \le \int_\Omega |(D+\A)\psi |^2,
$$
which gives the condition  \eqref{H-1} for any $\psi \in H^1(\Omega; \C)$.
\end{proof}

\begin{thm}\label{thm-N2}
Let $d\ge 3$.
Suppose $\Omega$ and $\B$ satisfy the same conditions as in Theorem \ref{thm-N1}.
For $F\in L^\infty(\Omega; \C)$, let $u\in H^1(\Omega; \C)$ be a weak solution of \eqref{NP-1} with $g=0$. Then
\begin{equation}\label{N2-0}
\| m(x, \B)^{\ell +2} u \|_{L^p(\Omega)}
\le C_\ell \| m(x, \B)^\ell  F \|_{L^p(\Omega)}
\end{equation}
for any $\ell \in \R$ and $1\le p\le \infty$, where $C_\ell $ depends on $d$, $\ell$,  $R_0$, 
the Lipschitz character of $\Omega$ and $C_0$ in \eqref{cod-1}.
\end{thm}

\begin{proof}
The proof is similar to that of Theorem \ref{thm-G2}.
\end{proof}


\section{Rellich estimates}\label{sec-R}

Let  $\Omega$ be a bounded Lipschitz domain in $\R^d$ and $\A\in C^1(\overline{\Omega}; \R^d)$.
Let $L=(L_1, \dots, L_d)$, where 
\begin{equation*}\label{L}
L_j =D_j +A_j
\end{equation*}
for $1\le j \le d$. 
It follows from integration by parts that if $f, g \in C^1(\overline{\Omega}; \C)$, 
\begin{equation}\label{parts}
\int_\Omega L_j f \cdot \overline{g}= \int_\Omega f \cdot \overline{L_j g} +\frac{1}{i} \int_{\partial\Omega}
n_j f \cdot \overline{g}.
\end{equation}

\begin{lemma}\label{lemma-r1}
Let $\alpha =(\alpha_1, \dots, \alpha_d)\in C^1(\R^d; \R^d)$.
Suppose $u\in C^2(\overline{\Omega}; \C)$ and  $(D+\A)^2 u =F$ in $\Omega$.
Then
\begin{equation}\label{re-0}
\aligned
\int_{\partial\Omega}
 \alpha_j n_j  | (D+\A) u|^2
&=2 \Re \int_{\partial\Omega}  \alpha_k  n_j  (D_j+A_j) u
\overline{   (D_k+A_k) u }\\
& +2 \Im \int_\Omega \alpha_k B_{kj} u \overline{ (D_j +A_j) u  }
  +\int_\Omega \text{\rm div}(\alpha)
|(D+\A)u|^2\\
 & -2 \Re \int_\Omega \partial_j \alpha_k (D_j +A_j) u \overline{ (D_k +A_k)u}
   + 2\Im \int_\Omega \alpha_k  F   \overline{ (D_k + A_k)u},
\endaligned
\end{equation}
where the repeated indices $j, k$ are summed from $1$ to $d$.
\end{lemma}

\begin{proof}

Using \eqref{parts}, we obtain 
$$
\aligned
\int_\Omega \alpha_k L_j u \overline{L_jL_k u}
&=\int_\Omega L_k (\alpha_k L_j u) \overline{L_j u}
-\frac{1}{i}\int_{\partial\Omega}
n_k\alpha_k |Lu|^2
+\int_\Omega \alpha_k L_j u \overline{[ L_j, L_k] u}\\
&=\int_\Omega \alpha_k L_kL_j u \overline{L_j u}
+ \frac{1}{i} \int_\Omega \text{\rm div} (\alpha) |L u|^2
-\frac{1}{i}\int_{\partial\Omega}
n_k\alpha_k |Lu|^2\\
 &\qquad\qquad\qquad
  +\int_\Omega \alpha_k L_j u \overline{[ L_j, L_k] u}\\
&=\int_\Omega \alpha_k L_j L_k u \overline{L_j u}
+\int_\Omega \alpha_k [ L_k, L_j ] u \overline{L_j u}
   + \frac{1}{i} \int_\Omega \text{\rm div} (\alpha) |L u|^2\\
 & \qquad\qquad\qquad  -\frac{1}{i}\int_{\partial\Omega}
n_k\alpha_k |Lu|^2
  +\int_\Omega \alpha_k L_j u \overline{[ L_j, L_k] u},
\endaligned
$$
where $[L_j, L_k]= L_j L_k -L_k L_j$.
This gives
\begin{equation}\label{re-1}
\aligned
2 \Im  \int_\Omega \alpha_k L_j u \overline{L_jL_k u}
 & = 2 \Im \int_\Omega \alpha_k [ L_k, L_j ] u \overline{L_j u}
 - \int_\Omega \text{\rm div} (\alpha) |L u|^2
+ \int_{\partial\Omega}
n_k\alpha_k |Lu|^2\\
&=-2 \Re \int_\Omega \alpha_k B_{kj} u \overline{L_j u}
 - \int_\Omega \text{\rm div} (\alpha) |L u|^2
+ \int_{\partial\Omega}
n_k\alpha_k |Lu|^2,
\endaligned
\end{equation}
where we have used the fact $[L_k, L_j ] = -i B_{kj}$.
Also note that  by \eqref{parts}, 
\begin{equation}\label{re-1a}
\aligned
\int_\Omega L_j^2 u \overline{\alpha_k L_k u}
&=\int_\Omega L_j u \overline{L_j (\alpha_k L_k u)}
+ \frac{1}{i} \int_{\partial\Omega} \alpha_k n_j L_j u \overline{ L_k u}\\
&=\int_\Omega \alpha_k L_j u \overline{L_j L_k u}
+ i \int_\Omega \partial_j \alpha_k  L_j u \overline{L_k u}
+ \frac{1}{i} \int_{\partial\Omega} \alpha_k n_j L_j u \overline{ L_k u}.\\
\endaligned
\end{equation}
From \eqref{re-1a}, we deduce that 
\begin{equation}\label{re-2}
2 \Im  \int_\Omega \alpha_k L_j u \overline{L_jL_k u}
=2\Im  \int_\Omega L_j^2 u \overline{\alpha_k L_k u}
-2 \Re \int_\Omega \partial_j \alpha_k L_j u \overline{L_k u}
+ 2\Re \int_{\partial\Omega}
\alpha_k n_j L_j u \overline{L_k u}.
\end{equation}
The equation \eqref{re-0} now follows readily from \eqref{re-1} and \eqref{re-2}.
\end{proof}

Let $T=(T_{jk})$, where
\begin{equation*}\label{T}
T_{jk} =n_j L_k -n_k L_j
\end{equation*}
for $1\le j, k \le d$.

\begin{lemma}\label{lemma-r2}
Let $\alpha \in C^1(\R^d; \R^d)$.
Suppose $u\in C^2(\overline{\Omega}; \C)$ and  $(D+\A)^2 u =F$ in $\Omega$.
Then
\begin{equation}\label{re-01}
\aligned
\int_{\partial\Omega}
 \alpha_j n_j  | (D+\A) u|^2
&=2 \Re \int_{\partial\Omega}  \alpha_k T_{kj} u 
\overline{   (D_j+A_j) u }\\
& -2 \Im \int_\Omega \alpha_k B_{kj} u \overline{ (D_j +A_j) u  }
  -\int_\Omega \text{\rm div}(\alpha)
|(D+\A)u|^2\\
 & +2 \Re \int_\Omega \partial_j \alpha_k (D_j +A_j) u \overline{ (D_k +A_k)u}
   -2\Im \int_\Omega \alpha_k  F   \overline{ (D_k + A_k)u},
\endaligned
\end{equation}
where the repeated indices $j, k$ are summed from $1$ to $d$.
\end{lemma}

\begin{proof}

Write the equation \eqref{re-0} as $I=J$.
Then \eqref{re-01} follows from the fact
$I=2I-I =2I-J$ as well as the observation that 
$$
\aligned
 & 2\int_{\partial\Omega}
 \alpha_j n_j  | (D+\A) u|^2
-2 \Re \int_{\partial\Omega}  \alpha_k  n_j  (D_j+A_j) u
\overline{   (D_k+A_k) u }\\
 & =2 \Re \int_{\partial\Omega}  \alpha_k T_{kj} u 
\overline{   (D_j+A_j) u }.
\endaligned
$$
\end{proof}

For the rest of  this section,  we shall further assume that $\Omega\subset \b(0, R_0)$ is a smooth domain for some $R_0=Cr_0>0$,
 where $C$ depends on the Lipschitz character of 
$\Omega$. We also assume
 that $\B$ satisfies \eqref{cod-1} for any $\b(x, r)\subset \b(0, 4R_0)$ and \eqref{cod-1a} for any $x\in \b(0, R_0)$.
 We emphases that
 although $\Omega$ is assumed to be smooth,  the dependence of constants $C$ on $\Omega$  
 is through $r_0$ and  its Lipschitz character.

\begin{lemma}\label{lemma-r3}
Suppose that $u\in C^2(\overline{\Omega}; \C)$ and
$(D+\A)^2 u =F$ in $\Omega$.
Then for any $\ell \in \R$, 
\begin{equation}\label{re-30}
\aligned
\int_{\partial\Omega}
|(D+\A)u |^2 m(x, \B)^\ell 
 & \le C \int_{\partial\Omega}
| n\cdot  (D+\A)  u|^2 m(x, \B)^\ell 
+ C \int_\Omega  |u|^2 m(x, \B)^{\ell+3} \\
 & \qquad
+ C \int_\Omega   |(D+\A) u|^2 m(x, \B)^{\ell+1}
+ C \int_\Omega |F|^2 m(x, \B)^{\ell-1},
\endaligned
\end{equation}
\begin{equation}\label{re-31}
\aligned
\int_{\partial\Omega}
|(D+\A)u |^2 m(x, \B)^\ell 
 & \le C \int_{\partial\Omega}
|  T  u|^2 m(x, \B)^\ell 
+ C \int_\Omega  | u|^2m(x, \B)^{\ell +3} \\
 & \qquad
 +C \int_\Omega |(D+\A)u|^2 m(x, \B)^{\ell+1} 
+ C \int_\Omega  |F|^2  m(x, \B)^{\ell-1},
\endaligned
\end{equation}
where $C$ depends on $\ell$, $\Omega$ and $C_0$ in \eqref{cod-1}.
\end{lemma}

\begin{proof}

Fix $x_0 \in \partial\Omega$.
Recall that $m(x_0, \B)> 2R_0^{-1}$ by \eqref{cod-1a}.
Let $r= c m(x_0, \B)^{-1}<  r_0$, where $c>0$ depends on the Lipschitz character of $\Omega$.
Choose $\alpha \in C_0^1 (\b(x_0, 2r); \R^d)$ such that
$\alpha \cdot n \ge 0$ on $\b(x_0, 2r)\cap \partial\Omega$,
$\alpha \cdot n \ge c_0>0$ on $\b(x_0, r)\cap \partial\Omega$,
$|\alpha |\le 1$ in $\b(x_0, 2r)$ and $|\nabla \alpha|\le C r^{-1}$ in $\b(x_0, 2r)$.
It follows from \eqref{re-0} that 
$$
\aligned
\int_{\b(x_0, r)\cap \partial\Omega} |(D+\A) u|^2
 & \le C \int_{\b(x_0, 2r)\cap \partial\Omega} |(D+\A) u| |n \cdot (D+\A) u|\\
 & \qquad+ C r^{-3} \int_{\b(x_0, 2r) \cap \Omega}  |u|^2
+ C r^{-1} \int_{\b(x_0, 2r)} |(D+\A) u|^2\\
 &\qquad  + C r  \int_{\b(x_0, 2r) \cap \Omega} |F|^2,
\endaligned
$$
where we have used the fact $ |\B (x)| \le C r^{-2}$ for $x\in \b(x_0, 2r)$.
We now multiply the inequality above by $m(x_0, \B)^{\ell+d-1}$ and integrate the resulting inequality in $x_0$
over $\partial\Omega$.
Using the fact that $m(x, \B) \approx m(y, \B)$ if $|x-y|< c m(x, \B)^{-1}$,
this leads to 
$$
\aligned
\int_{\partial\Omega}
|(D+\A)u|^2 m(x, \B)^\ell
& \le C \int_{\partial\Omega} |(D+\A) u| | n \cdot (D+\A) u| m(x, \B)^\ell\\
 & \qquad+ C \int_\Omega |u|^2  m(x, \B)^{\ell +3}
 + C \int_\Omega |(D+\A) u|^2 m(x, \B)^{\ell+1} \\
 &\qquad + C \int_\Omega |F|^2 m(x, \B)^{\ell-1}, 
\endaligned
$$
which yields \eqref{re-30} by using the Cauchy inequality. 
A similar argument, using \eqref{re-01}, gives \eqref{re-31}.
\end{proof}

\begin{lemma}\label{lemma-r4}
Under the same assumptions as in Lemma \ref{lemma-r3}, we have
\begin{equation}\label{re-40}
\aligned
\int_\Omega |(D+\A)u|^2 m(x, \B)^{\ell+1}
&\le C \int_\Omega |u|^2 m(x, \B)^{\ell+3}
+C \int_\Omega |F|^2 m(x, \B)^{\ell-1}\\
& \qquad
+ C \int_{\partial\Omega}
|u| | n \cdot (D+\A) u| m(x, \B)^{\ell+1},
\endaligned
\end{equation}
for any $\ell \in \R$, 
where $C$ depends on $\ell$.
\end{lemma}

\begin{proof}

Fix $x_0\in \partial\Omega$ and let $r=c m(x_0, \B)^{-1}< r_0$.
It follows from Remark \ref{C-re} that 
$$
\aligned
\int_{\b(x_0, r)\cap \Omega} |(D+\A) u|^2 m(x, \B)^{\ell +d}
& \le C \int_{\b(x_0, 2r)\cap \Omega} |u|^2 m(x, \B)^{\ell +d+2}\\
& \qquad+ C \int_{\b(x_0, 2r)\cap \Omega} |F|^2 m(x, \B)^{\ell+d-2}\\
& \qquad
+ C \int_{\b(x_0, 2r) \cap \partial\Omega}
|u| |n \cdot (D+\A) u| m(x, \B)^{\ell+d}.
\endaligned
$$
By integrating the inequality above in $x_0$ over $\partial\Omega$, we see that 
$$
\int_{\Omega_b} |(D+\A) u|^2 m(x, \B)^{\ell +1}
$$ 
is bounded by the right-hand side of \eqref{re-40}, where
$$
\Omega_b =\{ x\in \Omega: \ \text{\rm dist}(x, \partial\Omega)< c\,  m(x, \B)^{-1} \}.
$$
To handle the region $\Omega\setminus \Omega_b$, we use the fact that if $\b(x_0, 2t)\subset \Omega$, 
$$
\int_{\b(x_0, t)} |(D+\A) u|^2
\le \frac{C}{t^2} \int_{\b(x_0, 2t)} |u|^2 + C t^2 \int_{\b(x_0, 2t)} |F|^2.
$$
It follows that if $x_0\in \Omega\setminus \Omega_b$, then
$$
\aligned
 & \int_{\b(x_0, r)} |(D+\A) u|^2 m(x, \B)^{\ell-d+1}\\
 & \le C \int_{\b(x_0, 2r)} |u|^2 m(x, \B)^{\ell-d+3}
+ C \int_{\b(x_0, 2r)} |F|^2 m(x, \B)^{\ell-d-1},
\endaligned
$$
where $r=c m(x_0, \B)^{-1}$.
By integrating the inequality above in $x_0$ over $\Omega\setminus \Omega_b$,
we see that 
$$
\int_{\Omega\setminus \Omega_b} |(D+\A) u|^2 m(x, \B)^{\ell+1}
$$ 
is also bounded by the right-hand side of \eqref{re-40}.
\end{proof}

\begin{thm}\label{thm-r5}
Suppose $d\ge 3$.
Let $u\in C^2(\overline{\Omega}; \C)$.
Suppose $(D+\A)^2 u=F$ in $\Omega$.
Assume that either $u=0$ on $\partial\Omega$ or
$n \cdot (D+\A)u=0$ on $\partial\Omega$.
Then
\begin{equation}\label{r5a}
\int_\Omega |(D+\A) u|^2 m(x, \B)^\ell
+\int_\Omega |u|^2 m(x, \B)^{\ell +2}
\le C \int_\Omega |F|^2 m(x, \B)^{\ell-2}
\end{equation}
for any $\ell \in \R$, where $C$ depends on $\ell$.
\end{thm}

\begin{proof}
It follows from \eqref{re-40} that the left-hand side of \eqref{r5a} is bounded by
$$
C \int_\Omega |u|^2 m(x, \B)^{\ell +2}
+C\int_\Omega |F|^2 m (x, \B)^{\ell-2}.
$$
This yields  \eqref{r5a} by using Theorems \ref{thm-G2} and \ref{thm-N2} with $p=2$.
\end{proof}

\begin{lemma}\label{lemma-r5}
Let $u\in H^1(\Omega; \C)$. Then 
\begin{equation}\label{r5-0}
\int_{\partial\Omega}
|u|^2 m(x, \B)^\ell
\le C \int_\Omega |u|^2 m(x, \B)^{\ell+1}
+ C \int_\Omega |(D+\A)u|^2 m(x, \B)^{\ell -1}, 
\end{equation}
where $C$ depends on $\ell$ and $\Omega$.
\end{lemma}

\begin{proof}

Let $x_0\in \partial\Omega$ and $r=c m(x_0, \B)^{-1}< r_0$.
Then
$$
\int_{\b(x_0, r) \cap \partial\Omega}
|v|^2 
\le \frac{C}{r} \int_{\b(x_0, 2r) \cap \Omega} |v|^2
+ C r\int_{\b(x_0, 2r) \cap \Omega} |\nabla v|^2
$$
for any $v\in H^1(\Omega; \C)$.
By applying the inequality above to $v=v_\e =\sqrt{|u|^2 +\e^2}$ and letting $\e \to 0$, we obtain 
$$
\int_{\b(x_0, r) \cap \partial\Omega}
|u|^2 
\le \frac{C}{r} \int_{\b(x_0, 2r) \cap \Omega} |u|^2
+ C r\int_{\b(x_0, 2r) \cap \Omega} | (D+\A) u |^2.
$$
This leads to
$$
\aligned
\int_{\b(x_0, r)\cap \partial\Omega}
|u|^2 m(x, \B)^{\ell+d-1}
 & \le C \int_{\Omega\cap \b(x_0, 2r)}  |u|^2 m(x, \B)^{\ell +d}\\
& \qquad
+ C \int_{\Omega\cap \b(x_0, 2r)} | (D+\A) u|^2 m(x, \B)^{\ell+d-2}.
\endaligned
$$
By integrating the inequality above in $x_0$ over $\partial \Omega$, we obtain \eqref{r5-0}.
\end{proof}

\begin{lemma}\label{lemma-r6}
Under the same assumptions as in Lemma \ref{lemma-r3}, we have 
\begin{equation}\label{re-60}
\aligned
 & \int_{\partial\Omega}
|(D+\A)u |^2 m(x, \B)^\ell  +\int_{\partial\Omega} |u|^2 m(x, \B)^{\ell+2}\\
 & \le C \int_{\partial\Omega}
| n\cdot  (D+\A)  u|^2 m(x, \B)^\ell 
+ C \int_\Omega  |u|^2 m(x, \B)^{\ell+3} 
+ C \int_\Omega |F|^2 m(x, \B)^{\ell-1},
\endaligned
\end{equation}
\begin{equation}\label{re-61}
\aligned
\int_{\partial\Omega}
|(D+\A)u |^2 m(x, \B)^\ell 
 & \le C \int_{\partial\Omega}
|  T  u|^2 m(x, \B)^\ell 
+ C \int_{\partial\Omega}  | u|^2m(x, \B)^{\ell +2} \\
 & \qquad
 +C \int_\Omega |u|^2 m(x, \B)^{\ell+3} 
+ C \int_\Omega  |F|^2  m(x, \B)^{\ell-1},
\endaligned
\end{equation}
for any $\ell\in \R$, where $C$ depends on $\ell$, $C_0$ in \eqref{cod-1} and the Lipschitz character of  $\Omega$.
\end{lemma}

\begin{proof}

To see \eqref{re-60}, we use \eqref{re-30} and \eqref{r5-0} to obtain 
$$
\aligned
 & \int_{\partial\Omega}
|(D+\A)u |^2 m(x, \B)^\ell + \int_{\partial\Omega}  |u|^2 m(x, \B)^{\ell+2}\\
 & \le C   \int_{\partial\Omega}
| n\cdot  (D+\A)  u|^2 m(x, \B)^\ell 
+ C \int_\Omega  |u|^2 m(x, \B)^{\ell+3}\\
 & \qquad+ C \int_\Omega | (D+\A)u|^2 m(x, \B)^{\ell+1}  
+ C \int_\Omega |F|^2 m(x, \B)^{\ell-1}\\
&\le 
C   \int_{\partial\Omega}
| n\cdot  (D+\A)  u|^2 m(x, \B)^\ell 
+ C \int_\Omega  |u|^2 m(x, \B)^{\ell+3}\\
 & \qquad
+ C \int_\Omega |F|^2 m(x, \B)^{\ell-1}
+ C \int_{\partial\Omega}
|u| |n \cdot (D+\A) u| m(x, \B)^{\ell+1},
\endaligned
$$
where we have used \eqref{re-40} for the second inequality.
The inequality \eqref{re-60} now follows readily by applying the Cauchy inequality to the last integral.

The proof for \eqref{re-61} is similar.
By \eqref{re-31} and \eqref{re-40}, we obtain 
$$
\aligned
  \int_{\partial\Omega}
|(D+\A)u |^2 m(x, \B)^\ell 
 & \le C   \int_{\partial\Omega}
|  Tu |^2 m(x, \B)^\ell 
+ C \int_\Omega  |u|^2 m(x, \B)^{\ell+3}\\
 & \quad
+ C \int_\Omega |F|^2 m(x, \B)^{\ell-1}
+ C \int_{\partial\Omega} |u| |n \cdot (D+\A) u | m(x, \B)^{\ell+1}, 
\endaligned
$$
which yields \eqref{re-61} by applying the Cauchy inequality to the last integral.
\end{proof}

\begin{lemma}\label{lemma-r7}
Let $d\ge 3$.
Suppose $u\in C^1(\overline{\Omega}; \C)$ and  $(D+\A)^2 u=0$ in $\Omega$.
Then for any $\ell \in \R$, 
\begin{align}
\int_\Omega |u|^2 m(x, \B)^{\ell+3}
 & \le C \int_{\partial\Omega} |u|^2 m(x, \B)^{\ell+2},\label{re-70}\\
\int_{\Omega} |u|^2 m(x, \B)^{\ell+3}
 &\le C \int_{\partial\Omega}
|n \cdot (D+\A) u|^2 m(x, \B)^\ell,\label{re-71}
\end{align}
where $C$ depends on $\ell$ and $\Omega$.
\end{lemma}

\begin{proof}

We use a duality argument.
For $G\in C_0^\infty (\Omega; \C)$, let $v\in H^1(\Omega; \C)$ be the weak solution of
$(D+\A)^2 v=G$ in $\Omega$ with the Neumann condition  $n \cdot (D+\A) v=0$ on $\partial\Omega$.
Note that 
$$
\int_\Omega u \cdot \overline{G}
=\int_\Omega (D+\A) u \cdot \overline{(D+\A) v}
= i \int_{\partial\Omega}
n \cdot (D+\A) u\cdot  \overline{v}.
$$
It follows by the Cauchy inequality  that
\begin{equation}\label{re-74}
\aligned
\Big| \int_\Omega u \cdot \overline{G} \Big|
& \le \left(\int_{\partial\Omega} | n \cdot (D+\A) u|^2 m(x, \B)^\ell \right)^{1/2}
\left(\int_{\partial\Omega}
|v|^2 m(x, \B)^{-\ell} \right)^{1/2}.
\endaligned
\end{equation}
We will show that 
\begin{equation}\label{re-72}
\int_{\partial\Omega} |v|^2 m(x, \B)^{-\ell}
\le C \int_\Omega 
|G|^2 m(x, \B)^{-\ell-3}, 
\end{equation}
which, together with \eqref{re-74},  yields \eqref{re-71} by duality.

Since $\Omega$ is smooth and $G\in C_0^\infty(\Omega; \C)$, we have $v\in C^2(\overline{\Omega}; \C)$. 
By using  Lemma \ref{lemma-r6}, we obtain 
$$
\int_{\partial\Omega}
|v|^2 m(x, \B)^{-\ell}
\le C \int_\Omega |v|^2 m(x, \B)^{-\ell+1}
+ C \int_\Omega |G|^2 m(x, \B)^{-\ell-3}.
$$
This, together with \eqref{N2-0} with $p=2$, gives \eqref{re-72}.

The proof of \eqref{re-70} is similar.
For $\Psi\in C_0^\infty(\Omega; \C)$, let $w\in H_0^1(\Omega; \C)$
be the weak solution of $(D+\A)^2 w=\Psi$ in $\Omega$ with the Dirichlet condition $w=0$ on $\partial\Omega$.
Then
$$
\int_\Omega u \cdot \overline{\Psi}
=i\int_{\partial\Omega} u \overline{n\cdot (D+\A) w}.
$$
Hence, by the Cauchy inequality, 
\begin{equation}\label{re-75}
\Big|\int_\Omega u\cdot \overline{\Psi} \Big|
\le \left(\int_{\partial\Omega}
|u|^2 m(x, \B)^{\ell+2}\right)^{1/2}
\left(\int_{\partial\Omega}
| n \cdot (D+\A)w|^2 m(x, \B)^{-\ell -2} \right)^{1/2}.
\end{equation}
We claim that
\begin{equation}\label{claim-1}
\int_{\partial\Omega}
|(D+\A) w|^2 m(x, \B)^{-\ell-2}
\le C \int_\Omega | \Psi|^2 m(x, \B)^{-\ell-3},
\end{equation}
which, together with \eqref{re-75},  leads to \eqref{re-70} by duality.

Finally, to prove \eqref{claim-1}, we apply the estimate \eqref{re-61}.
Since $\Omega$ is smooth and $\Psi\in C_0^\infty(\Omega; \C)$, we have $w\in C^2(\overline{\Omega}; \C)$.
Observe that  since $w=0$ on $\partial\Omega$, we have $ Tw=0$ on $\partial\Omega$.
As a result, 
$$
\int_{\partial\Omega}
|(D+\A) w|^2 m(x, \B)^{-\ell-2}
\le C \int_\Omega |w|^2 m(x, \B)^{-\ell+1}
+ C \int_\Omega \Psi|^2 m(x, \B)^{-\ell-3}.
$$
This, together with the estimate \eqref{G2-0} with $p=2$, gives \eqref{claim-1}
\end{proof}

We are now ready to prove the main results of this section.

\begin{thm}\label{re-theorem}
Let $d\ge 3$.
Suppose  $u\in C^2(\overline{\Omega}; \C)$ and $(D+\A)^2 u =0 $ in $\Omega$.
Then for any $\ell \in \R$, 
\begin{equation}\label{re-a}
\aligned
 &  \int_{\partial\Omega}
|(D+\A)u |^2 m(x, \B)^\ell
+\int_{\partial\Omega}
|u|^2 m(x, \B)^{\ell+2}
+\int_\Omega |u|^2 m(x, \B)^{\ell+3}\\
 &\qquad\qquad
  +\int_\Omega |(D+\A) u|^2 m(x, \B)^{\ell+1}
 \le C \int_{\partial\Omega}
|n \cdot (D+\A) u|^2 m(x, \B)^\ell, 
\endaligned
\end{equation}
\begin{equation}\label{re-b}
\aligned
 &  \int_{\partial\Omega}
|(D+\A)u |^2 m(x, \B)^\ell
+\int_\Omega |u|^2 m(x, \B)^{\ell+3}
  +\int_\Omega |(D+\A) u|^2 m(x, \B)^{\ell+1}\\
& \qquad \le C \int_{\partial\Omega}
| T u|^2 m(x, \B)^\ell
+ C \int_{\partial\Omega}
|u|^2 m(x, \B)^{\ell+2}, 
\endaligned
\end{equation}
where $C$ depends on $\ell$ and $\Omega$.
\end{thm}

\begin{proof}

We give the proof of \eqref{re-a}. A similar argument yields \eqref{re-b}.

First, by \eqref{re-60} and \eqref{re-71}, we obtain 
\begin{equation}\label{re-a1}
\int_{\partial\Omega} |(D+\A) u|^2 m(x, \B)^\ell
+\int_{\partial\Omega} |u|^2 m(x, \B)^{\ell+2} 
\le C \int_{\partial\Omega} |n \cdot (D+\A) u|^2 m(x, \B)^\ell.
\end{equation}
Next,  by \eqref{re-40} and \eqref{re-71},
 the third and fourth terms in the left-hand side of \eqref{re-a} are bounded by
 $$
 C \int_{\partial\Omega}
 | n \cdot (D+\A) u|^2 m(x, \B)^\ell
 + C \int_{\partial\Omega}
 |u| |n\cdot (D+\A) u|^2 m(x, \B)^{\ell+1}.
 $$
 In view of \eqref{re-a1}, the integrals above are bounded by the right-hand side of \eqref{re-a}.
\end{proof}


\section{Nontangential maximal function estimates}\label{sec-NM}

Throughout this section,  we assume that  $d\ge 3$ and 
$\Omega\subset \b(0, R_0)$ is a smooth domain for some $R_0=Cr_0>0$,
 where $C$ depends on the Lipschitz character of 
$\Omega$. We also assume
 that $\B$ satisfies \eqref{cod-1} for any $\b(x, r)\subset \b(0, 4R_0)$ and \eqref{cod-1a} for any $x\in \b(0, R_0)$.
 As before,    the dependence of constants $C$ on $\Omega$  
 is through $r_0$ and  its Lipschitz character.
 
 \begin{lemma}\label{lemma-mx1}
 Let $u\in C^1(\overline{\Omega}; \C)$ and  $w(x)=u(x) m(x, \B)$.
Let  $\mathcal{M}(w)$ be defined by \eqref{m-max}.
  Then
 \begin{equation}\label{n8-0}
 \aligned
\int_{\partial\Omega}
|\mathcal{M}(w)|^2
 & \le  C \int_{\partial\Omega} | w|^2 + C \int_\Omega |(D+\A)  u(x)|^2 m(x, \B)\, dx\\
&\qquad\qquad
+ C \int_\Omega | u(x) |^2 m(x, \B)^3\, dx,
\endaligned
\end{equation}
where $C$ depends on $\Omega$ and $C_0$ in \eqref{cod-1}. 
 \end{lemma}
 
 \begin{proof}
 
 Fix $x_0\in \partial\Omega$.
 By a change of the coordinate system, we may assume that $x_0=0$ and
 $$
 \Omega\cap \b(0, r_0)
 =\left\{ (x^\prime, x_d)\in \R^d: x^\prime \in \R^{d-1} \text{ and } x_d> \phi(x^\prime) \right\}
 \cap \b(0, r_0), 
 $$
 where $\phi: \R^{d-1} \to \R$ is a Lipchitz function with $\phi(0)=0$ and $\|\nabla \phi\|_\infty \le M_0$.
 For $x=(x^\prime, \phi(x^\prime)) \in \partial\Omega$ with $|x|\le c_0r_0$, define
 \begin{equation*}
 \mathcal{M}_1 (w) (x)
 =\sup \big\{ |w (x^\prime, t)|: \phi(x^\prime)< t< c r_0 \big\}.
 \end{equation*}
 It is not hard to see that 
 \begin{equation*}
 \mathcal{M}(w)  (x) \le C \mathcal{M}_{\partial\Omega} \big( \mathcal{M}_1 (w) \big) (x),
 \end{equation*}
 where $\mathcal{M}_{\partial\Omega}$ denotes the Hardy-Littlewood maximal operator on $\partial\Omega$, defined by 
 \begin{equation*}
 \mathcal{M}_{\partial\Omega} (g) (x)
 =\sup \Big\{ \fint_{\b(x, r)\cap \partial\Omega} |g|:  0< r<c r_0 \Big\}
 \end{equation*}
 for $x\in \partial\Omega$. As a result, by the $L^2$ boundedness of $\mathcal{M}_{\partial\Omega}$,
 \begin{equation}\label{n8-1}
 \int_{\b(0, cr_0)\cap\partial \Omega} |\mathcal{M}(w)|^2
 \le C \int_{\b(0, c_1 r_0)\cap \partial\Omega} |\mathcal{M}_1 (w)|^2,
 \end{equation}
 where  $c_1=C c_0$.
 
 Next, note that for $x=(x^\prime, \phi(x^\prime))$ with $|x|\le cr_0$,
 $$
  \{ \mathcal{M}_1 (w) (x)\}^2  \le |w (x^\prime, \phi(x^\prime)|^2
  +2 \int_{\phi(x^\prime)}^{cr_0} |\partial_t w(x^\prime, t)| |w(x^\prime, t)| dt.
  $$
  This, together with \eqref{n8-1}, leads to 
  $$
  \int_{\b(0, c r_0)\cap \partial\Omega} | \mathcal{M} (w)|^2
  \le C  \int_{\b(0, c r_0)\cap \partial\Omega} |w|^2
  + C \int_{\b(0, c_2 r_0)\cap \Omega}
  |\nabla w| |w|.
  $$
  By a covering argument, we obtain
  \begin{equation}\label{n8-1a}
  \aligned
  \int_{\partial\Omega} |\mathcal{M}(w)|^2
   & \le C \int_{\partial\Omega} |w|^2 + C \int_\Omega |\nabla w| |w|\\
   & \le  C \int_{\partial\Omega} |w|^2
   + C \int_\Omega |\nabla w|^2 m(x, \B)^{-1} + C \int_\Omega |w|^2 m(x, \B)
   \endaligned
  \end{equation}
  for any $w\in C^1(\overline{\Omega}; \C)$, where we have used the Cauchy inequality for the last step.
  We now  apply  \eqref{n8-1a} to the function
$w_\e=\sqrt{|w|^2 +\e^2}$ and then let $\e\to 0$. 
Using $|\nabla w_\e|\le |(D+\A) w|$, we deduce that 
 \begin{equation}\label{n8-1aa}
  \aligned
  \int_{\partial\Omega} |\mathcal{M}(w)|^2
   & \le  C \int_{\partial\Omega} |w|^2
   + C \int_\Omega | (D+\A) w|^2 m(x, \B)^{-1} + C \int_\Omega |w|^2 m(x, \B).
   \endaligned
  \end{equation}

Finally, we apply \eqref{n8-1aa}  to $w=u \widetilde{m}$, where $\widetilde{m}$ is a function in $C^1(\b(0, 2R_0); \R)$
  with the properties that $\widetilde{m}(x)\approx m(x, \B)$ for $x\in \b(0, 2R_0)$ and
  $|\nabla \widetilde{m}(x)|\le C m(x, \B)^2$ \cite {Shen-1996}.
 It follows that 
  \begin{equation}\label{n8-1b}
  C \int_{\partial\Omega} | \mathcal{M} (u\widetilde{m})|^2
    \le  C \int_{\partial\Omega} | u \widetilde{m}|^2
   + C \int_\Omega |(D+\A)u |^2 m(x, \B) + C \int_\Omega |u|^2 m(x, \B)^3
\end{equation}
for any $u\in C^1(\overline{\Omega}; \C)$.
This gives \eqref{n8-0}.
 \end{proof}

\begin{thm}\label{thm-n8-1}
Let $u\in C^2(\overline{\Omega}; \C)$ be a solution of the Neumann problem,  
\begin{equation}\label{NP-N}
(D+\A)^2 u= 0 \quad \text{ in } \Omega \quad \text{ and } \quad
n\cdot (D+\A)u=g \quad \text{ on } \partial\Omega.
\end{equation}
Let $v(x)= |(D+\A) u(x)| + m(x, \B) |u(x)| $. Then
\begin{equation}\label{m-est-1}
\| ( v)^* \|_{L^2(\partial\Omega)} 
\le C \| g \|_{L^2(\partial\Omega)}, 
\end{equation}
where $C$ depends on $C_0$ in \eqref{cod-1} and $\Omega$.
\end{thm}

\begin{proof}

In view of Corollary \ref{cor-M}, it suffices to prove that 
\begin{equation}\label{n8-2}
\| \mathcal{M} (v) \|_{L^2(\partial\Omega)} \le C \| g \|_{L^2(\partial\Omega)}, 
\end{equation}
where $\mathcal{M}(v)$ is defined by \eqref{m-max}.

Let $w=(D_k +A_k)u$ for $1\le k \le d$. Then $(D+\A)^2 w =F$, where
\begin{equation}\label{F1}
|F|\le 2 |\B| |(D+\A) u| + |\nabla \B| |u|.
\end{equation}
Let $w=w_1 + w_2$, where $(D+\A)^2 w_1 =0$ in $\Omega$ and $w_1 =w$ on $\partial\Omega$.
By Theorem \ref{thm-p1}, 
\begin{equation}\label{n8-2a}
\aligned
\|(w_1)^*\|_{L^2(\partial\Omega)}
&\le C \| w_1 \|_{L^2(\partial\Omega)}  = C \| (D_k+A_k ) u\|_{L^2(\partial\Omega)}\\
& \le C \| g \|_{L^2(\partial\Omega)}, 
\endaligned
\end{equation}
where we have used 
the Rellich estimate \eqref{re-a} for the last inequality.

Next, note that $(D+\A)^2 w_2=F$ in $\Omega$ and $w_2=0$ on $\partial\Omega$.
It follows by Theorem \ref{thm-r5} that
$$
\aligned
 & \int_\Omega |(D+\A) w_2|^2 m(x, \B)^{-1}
+\int_\Omega |w_2|^2 m(x, \B)\\
&\qquad \le C \int_\Omega |F|^2 m(x, \B)^{-3}\\
&\qquad\le C \int_\Omega |\B|^2 |(D+\A) u|^2 m(x, \B)^{-3}
+ C \int_\Omega |\nabla \B|^2 |u|^2 m(x, \B)^{-3},
\endaligned
$$
where we have used \eqref{F1} for the last inequality.
Since $|\B (x)|\le C m(x, \B)^2$ and $|\nabla \B(x)| \le C m(x, \B)^3$, we obtain 
$$
\aligned
 & \int_\Omega |(D+\A) w_2|^2 m(x, \B)^{-1}
+\int_\Omega |w_2|^2 m(x, \B) \\
&\qquad\le C \int_\Omega |(D+\A)u|^2 m(x, \B)
+ C \int_\Omega |u|^2 m(x, \B)^3\\
&\qquad \le C \int_{\partial\Omega} |g|^2,
\endaligned
$$
where we have used the Rellich estimate \eqref{re-a} with $\ell=0$ for the last step.
In view of Lemma \ref{lemma-mx1}, this gives
\begin{equation}\label{n8-2b}
\| \mathcal{M}(w_2)\|_{L^2(\partial\Omega)}
\le C \| g \|_{L^2(\partial\Omega)}.
\end{equation}

Finally, to prove \eqref{n8-2}, note that
$$
\aligned
\| \mathcal{M}(w)\|_{L^2(\partial\Omega)}
& \le C \| (w_1)^* \|_{L^2(\partial\Omega)}
+ C \| \mathcal{M}(w_2) \|_{L^2(\partial\Omega)}\\
& \le C \| g \|_{L^2(\partial\Omega)},
\endaligned
$$
where we have used \eqref{n8-2a} and \eqref{n8-2b} for the second inequality.
Moreover, if $\widetilde{v}(x)= m(x, \B) u(x)$, then
$$
\aligned
\| \mathcal{M}(\widetilde{v})\|_{L^2(\partial\Omega)}
 & \le  C \left\{ \| \widetilde{v}\|_{L^2(\partial\Omega)}
+ \| m(x, \B)^{1/2} (D+\A) u\|_{L^2(\Omega)}
+ \| m(x, \B)^{3/2} u \|_{L^2(\Omega)}\right \}\\
& \le C \| g \|_{L^2(\partial\Omega)},
\endaligned
$$
where we have used \eqref{n8-0} for the first inequality and \eqref{re-a} with $\ell=0$ for the second.
\end{proof}

\begin{thm}\label{thm-n8-2}
Let $u\in C^2(\overline{\Omega}; \C)$ be a solution of the Dirichlet problem,  
\begin{equation}\label{N8-D}
(D+\A)^2 u= 0 \quad \text{ in } \Omega \quad \text{ and } \quad
u=f  \quad \text{ on } \partial\Omega.
\end{equation}
Let $v(x)=|(D+\A)u (x)| + m(x, \B) |u(x)|$. Then
\begin{equation}\label{m-est-2}
\aligned
 \| (v)^*\|_{L^2(\partial\Omega)}
 &  \le C\left\{  \| Tf \|_{L^2(\partial\Omega)} + \| m(x, \B) f \|_{L^2(\partial\Omega)} \right\}, 
 \endaligned
\end{equation}
where $C$ depends on $C_0$ in \eqref{cod-1} and $\Omega$.
\end{thm}

\begin{proof}

The proof is similar to that of Theorem \ref{thm-n8-1}, using \eqref{re-b} in the place of \eqref{re-a}.
\end{proof}


\section{An approximation argument}\label{sec-A}

In this section we use an approximation argument to extend the results in the last section 
to Lipschitz domains in $\R^d$, $d\ge 3$.
We assume that $\Omega\subset \b(0, R_0)$ is a Lipschitz domain in $\R^d$, 
where $R_0=C r_0$ and $C$ depends on the Lipschitz character of $\Omega$.
As in the last section, 
we also assume
 that $\B$ satisfies \eqref{cod-1} for any $\b(x, r)\subset \b(0, 4R_0)$ and \eqref{cod-1a} for any $x\in \b(0, R_0)$.
 
\begin{lemma}\label{lemma-app}
Let $\Omega$ be a bounded Lipschitz domain in $\R^d$.
There exists a sequence of smooth domains $\{ \Omega_\ell \}$
with uniform Lipschitz characters such that
$\Omega_1\subset \Omega_2\subset \cdots \subset \Omega_\ell \cdots \subset \Omega$ with the following properties:

\begin{itemize}

\item

There exists a sequence of homeomorphisms $\Lambda_\ell : \partial \Omega
\to \partial \Omega_\ell$ such that  $\Lambda_\ell (x) \to x$ uniformly on $\partial\Omega$ as $\ell \to \infty$. Moreover, 
$$
|\Lambda_\ell (x) - x|\le C \,  \text{\rm dist}(\Lambda_\ell (x) ,  \partial\Omega)
$$
 for any $x\in \partial\Omega$.
 
 \item
 
 The unit outward normal to $\partial\Omega_\ell$,
 $n(\Lambda_\ell (x))$ converges to $n(x)$ for a.e.~$x\in \partial\Omega$.
 
 \item
 
 There are positive functions $\omega_\ell$ on $\partial\Omega$ such that
 $0<c \le \omega_\ell\le C$ uniformly in $\ell$, $\omega_\ell \to 1$ a.e.~as $\ell \to \infty$, and
 $$
 \int_E \omega_\ell dx = \int_{\Lambda_\ell (E)} dx
 $$
 for any measurable $E\subset \partial\Omega$.
 
 \item
 
 There exists $\alpha \in C_0^\infty(\R^d; \R^d)$ such that 
 $$
 \alpha (x) \cdot n(x)\ge c>0 \quad \text{ for any } x\in \partial\Omega_\ell \text{ and any } \ell.
 $$
 \end{itemize}
\end{lemma}

\begin{proof}

See \cite{Verchota-1985}.
\end{proof}

\begin{lemma}\label{lemma-q}
Let  $\A\in C^1(\overline{\Omega}; \R^d)$, where  $\Omega$ is a bounded Lipschitz domain in $\R^d$.
Suppose that  $u \in H^1(\Omega; \C)$ is  a weak solution of the Neumann problem \eqref{NP-N} for
some  $g\in L^2(\partial\Omega; \C)$,  or a weak solution of the Dirichlet problem \eqref{N8-D}
for some $f\in H^1(\partial\Omega; \C)$.
 Then
\begin{equation}\label{D-11}
\int_{\partial\Omega}  | \nabla u (\Lambda_\ell (x)) -\nabla u (x)|^2 dx
+\int_{\partial\Omega}
|u(\Lambda_\ell (x)) - u(x)|^2 dx
\to 0
\end{equation}
as $\ell \to \infty$, where $\Lambda_\ell$ is given by Lemma \ref{lemma-app}.
\end{lemma}

\begin{proof}

We give the proof for the Neumann problem \eqref{NP-N}. The proof for \eqref{N8-D} is similar.

Since $\A\in C^1(\overline{\Omega}; \R^d)$ and $u\in H^1(\Omega; \C)$, we may rewrite \eqref{NP-N} as
$$
-\Delta u =F \quad \text{ in } \Omega \quad \text{ and } \quad \frac{\partial u}{\partial n} =\widetilde{g}
$$
for some $F\in L^2(\Omega; \C)$ and $\widetilde{g}\in L^2(\partial\Omega; \C)$.
Let $u=v + w$, where $w=\Gamma * (F\chi_\Omega)$ and $\Gamma$ denotes the fundamental solution for 
$-\Delta$ in $\R^d$. Note that $w\in H^2(\R^d; \C)$.
It  follows that
\begin{equation}\label{D-12}
\aligned
& \int_{\partial\Omega}  | \nabla w (\Lambda_\ell (x)) -\nabla w (x)|^2 dx
+\int_{\partial\Omega}
|w(\Lambda_\ell (x)) - w(x)|^2 dx\\
&\qquad\qquad
 \le C \int_{\Omega\setminus \Omega_\ell}
\left\{ |\nabla^2 w |^2 +|\nabla  w|^2 \right\}
\to 0,
\endaligned
\end{equation}
as $\ell \to \infty$.
Also observe that  $v \in H^1(\Omega; \C)$ is harmonic in $\Omega$ and $\frac{\partial v}{\partial n}
=\widetilde{g}-n \cdot \nabla w \in L^2(\partial\Omega; \C)$.
As a result, $v$ and $\nabla v$ have nontangential limits a.e.~on $\partial\Omega$ and
$(v)^* + (\nabla v)^* \in L^2(\partial\Omega)$ \cite{DK-1980, DK-1981}. 
By the dominated convergence theorem,
\begin{equation*}
\int_{\partial\Omega}  | \nabla v (\Lambda_\ell (x)) -\nabla v (x)|^2 dx
+\int_{\partial\Omega}
|v(\Lambda_\ell (x)) - v(x)|^2 dx \to 0
\end{equation*}
as $\ell \to \infty$, which, together with \eqref{D-12}, gives \eqref{D-11}.
\end{proof}

\begin{thm}\label{thm-n9-1}
Let $u\in H^1(\Omega; \C)$ be a solution of the Neumann problem \eqref{NP-N}
for some $g\in L^2(\partial\Omega; \C)$.
Let $v(x)=|(D+\A)u(x)| + m(x, \B) |u(x)|$. Then
\begin{equation}\label{n9-1a}
\| (v)^*\|_{L^2(\partial \Omega)}
+ \| m(x, \B)^{1/2}  v \|_{L^2(\Omega)}
\le C \| g \|_{L^2(\partial\Omega)},
\end{equation}
where $C$ depends on $C_0$ in \eqref{cod-2}, $r_0$ and the Lipschitz character of $\Omega$.
\end{thm}

\begin{proof}

Let $u\in H^1(\Omega; \C)$ be a weak solution of \eqref{NP-N} with $g\in L^2(\partial\Omega; \C)$.
Let $\{ \Omega_\ell \}$ be a sequence of smooth domains, given by Lemma \ref{lemma-app}.
Since $\overline{\Omega}_\ell \subset \Omega$ and $\A\in C^2(\overline{\Omega}; \R^d)$,
  we have $u\in C^2(\overline{\Omega}_\ell; \C)$.
This allows us to apply Theorem \ref{thm-n8-1} and \eqref{re-a}  to obtain 
\begin{equation}\label{n9-1b}
\| (v)^*_\ell \|_{L^2(\partial\Omega_\ell)}
+   \|  m(x, \B)^{1/2} v \|_{L^2(\Omega_\ell)}
\le C \| n \cdot (D+\A) u \|_{L^2(\partial\Omega_\ell)},
\end{equation}
where we have used $(v)^*_\ell$ to denote the nontangential maximal function of
$v$ for the domain $\Omega_\ell$.
Thanks to Lemma \ref{lemma-app}, 
 the constant $C$ in  \eqref{n9-1b} depends only on $r_0$, the Lipschitz character of $\Omega$,  and $C_0$ in \eqref{cod-1}.
In particular, the estimate \eqref{n9-1b} is uniform with respect to $\ell$.
Using the observation  $\| (v)^*_\ell \|_{L^2(\partial\Omega_\ell)}
\le C \|(v)^*_m \|_{L^2(\partial\Omega_m)}$ for $\ell < m$, 
we see that 
\begin{equation*}
\| (v)^*_\ell  \|_{L^2(\partial\Omega_\ell)}
+   \|  m(x, \B)^{1/2} v \|_{L^2(\Omega_\ell)}
\le C \| n \cdot (D+\A) u \|_{L^2(\partial\Omega_m)}
\end{equation*}
for $\ell< m$.
We now let $m \to \infty$ in the inequality above.
Note that by Lemma \ref{lemma-q},
$$
 \| n \cdot (D+\A) u \|_{L^2(\partial\Omega_m)} \to \| g \|_{L^2(\partial\Omega)}.
 $$
Thus, 
\begin{equation*}
\| (v)^*_\ell  \|_{L^2(\partial\Omega_\ell)}
+   \|  m(x, \B)^{1/2} v \|_{L^2(\Omega_\ell)}
\le C  \| g \|_{L^2(\partial\Omega)}.
\end{equation*}
By letting $\ell  \to \infty$ in the inequality above,  we obtain \eqref{n9-1a}.
\end{proof}

\begin{thm}\label{thm-d9-1}
Let $u\in H^1(\Omega; \C)$ be a solution of the Dirichlet problem \eqref{N8-D}
for some $f\in H^1(\partial\Omega; \C)$.
Let $v(x)=|(D+\A)u(x)| + m(x, \B) |u(x)|$. Then
\begin{equation}\label{d9-1a}
\| (v)^*\|_{L^2(\partial \Omega)}
+ \| m(x, \B)^{1/2}  v \|_{L^2(\Omega)}
\le C\left\{  \| Tf \|_{L^2(\partial\Omega)} + \| m(x, \B) f \|_{L^2(\partial\Omega)} \right\},
\end{equation}
where $C$ depends on $C_0$ in \eqref{cod-2}, $r_0$ and the Lipschitz character of $\Omega$.
\end{thm}

\begin{proof}

The proof is similar to that of Theorem \ref{thm-n9-1}, using Theorem \ref{thm-n8-2} in the place of Theorem \ref{thm-n8-1}. 
\end{proof}


\section{Proof of Theorems \ref{main-thm-1} and \ref{main-thm-2}}\label{sec-F}

In this section, with Theorems \ref{thm-n9-1} and \ref{thm-d9-1} at our disposal, 
we use a localization argument to complete the proofs
of our main results.

Let $\Omega$ be a bounded Lipschitz domain in $\R^d$, $d\ge 2$.
Assume $\A\in C^\infty (\overline{\Omega}; \R^d)$ and
$\B=\nabla \times \A$ is of finite type on $\overline{\Omega}$.
It follows that  there  exist  an  integer $\kappa\ge 0$ and $c_0>0$ such that 
\begin{equation}\label{H-9}
\sum_{|\alpha|\le \kappa} |\partial^\alpha \B(x)| \ge c_0
\end{equation}
for any $x\in \overline{\Omega}$.
By extension we may assume $\B\in C^{\kappa+1}(\R^d; \R^{d\times d})$ and
$\| \B \|_{C^{\kappa+1} (\R^d)} \le  C \| \B \|_{C^{\kappa+1}(\overline{\Omega})}$ (there is no need to extend $\A$).
By compactness we may further assume that \eqref{H-9} holds for any $x\in \R^d$ with dist$(x, \Omega)< 2r_0$.

Let  $\widetilde{\Omega}=\{ x\in \R^d: \text{\rm dist}(x, \Omega)< r_0 \}$.

\begin{lemma}\label{lemma-9-1}
Suppose \eqref{H-9} holds for any $x\in \R^d$ with $\text{\rm dist}(x, \Omega)< 2r_0$. Then
\begin{equation}\label{9-1-0}
\sup_{\b(x, r)} |\partial^\alpha \B|
\le \frac{C_1}{r^{|\alpha|}}
\fint_{\b(x, r)} |\B|
\end{equation}
for any $\b (x, r) \subset \widetilde{\Omega}$ and 
 any $\alpha$ with $|\alpha|\le \kappa$.
  Moreover, 
 \begin{equation}\label{9-1-0a}
 \sup_{\b(x, r)} |\B|\ge c_1 r^\kappa
 \end{equation}
 for any $x\in \overline{\Omega} $ and $0< r< r_0$.
 The constants $C_1, c_1>0$ depend on $\Omega$, $(\kappa, c_0)$ in \eqref{H-9} and $\|\B \|_{C^{\kappa+1}(\overline{\Omega})}$.
 \end{lemma}

\begin{proof}

Let $\b (y, r) \subset  \{ x\in \R^d: \text{\rm dist}(x, \Omega)< 2r_0\}$.
Let $\P_y$ denote the $\kappa^{th}$ (matrix-valued) Taylor polynomial of $\B$ at $y$.
Then, if $|\alpha|\le \kappa$, 
$$
|\partial^\alpha \P_y (x)-\partial^\alpha \B (x)|\le C |x-y|^{\kappa+1-|\alpha|}
$$
for any $x\in \b (y, r)$, where $C$ depends on $\kappa$ and $\|\B\|_{C^{\kappa+1}({\R^d})}$.
It follows that
\begin{equation}\label{9-1-1}
\aligned
\sup_{\b(y, r)} |\partial^\alpha \B|
& \le \sup_{\b (y, r)} |\partial^\alpha \P_y | + C r^{\kappa+1 -|\alpha|}\\
& \le \frac{C}{r^{|\alpha|}} \fint_{\b (y, r)} |\P_y| + C r^{\kappa+1 -|\alpha|}\\
&\le   \frac{C}{r^{|\alpha|}} \fint_{\b (y, r)} |\B| + C r^{\kappa+1 -|\alpha|}
\endaligned
\end{equation}
for any $\alpha$ with $|\alpha|\le \kappa$.
Also note that
\begin{equation}\label{9-1-1a}
\aligned
\fint_{\b (y, r)} |\B|
 & \ge \fint_{\b (y, r)} |\P_y| - C r^{\kappa+1}\\
& \ge c\sum_{|\alpha|\le \kappa} |\partial^\alpha \B (y)| r^{|\alpha|} - C r^{\kappa+1}\\
& \ge c r^{\kappa} - C r^{\kappa+1},
\endaligned
\end{equation}
where we have used \eqref{H-9} for the last step.
As a result, we see that
\begin{equation}\label{9-1-1b}
\fint_{\b (y, r)} |\B| \ge c r^\kappa
\end{equation}
if $0<r<c$ and $c$ is sufficiently small.
This gives \eqref{9-1-0a}.
Together with \eqref{9-1-1}, it also yields  \eqref{9-1-0} under the additional condition  $0< r<c$.
Note that the inequality \eqref{9-1-0} with $\alpha=0$  and $0< r< c$ implies the doubling condition,
$$
\int_{\b(x, 2r)} |\B|\le C \int_{\b(x, r)} |\B|
$$
for any $\b(x, 4r)\subset \{ x\in \R^d:  \text{\rm dist}(x, \Omega)< 2r_0 \}$.
The general case for \eqref{9-1-0}  follows by a covering argument. 
\end{proof}

\begin{remark}\label{re-9-1}
{\rm 

Suppose \eqref{H-9} holds for any $x\in \R^d$ with dist$(x, \partial\Omega)< 2r_0$.
Also assume that 
\begin{equation}\label{H-9a}
\sup_{\b (x, r_0)} |\B| \ge 2  r_0^{-2}
\qquad \text{ for any } x\in \overline{\Omega}.
\end{equation}
It follows by definition that $m(x, \B) > r_0^{-1}$ for any $x\in \overline{\Omega}$.
By letting $r=m(x, \B)^{-1}< r_0$ in \eqref{9-1-0}, we see that 
$$
|\partial^\alpha \B(x)|\le C m(x, \B)^{|\alpha|+2}
$$
for any $x\in \overline{\Omega}$ and $|\alpha|\le \kappa$.
As a result, we obtain 
\begin{equation}\label{9-m}
m(x, \B) \ge c \sum_{|\alpha|\le \kappa} |\partial^\alpha \B (x)|^{\frac{1}{|\alpha|+2}}
\end{equation}
for any $x\in \overline{\Omega}$.
By definition it is not hard to see that $m(x, \B)\le C \| \B\|^{1/2}_{L^\infty(\widetilde{\Omega})}$.
}
\end{remark}

\begin{proof}[Proof of Theorem \ref{main-thm-1}]

We first consider the case $d\ge 3$.
Rewrite \eqref{NP0} as
\begin{equation*}
(D+ \beta \A)^2 u = 0 \quad 
\text{ in } \Omega \quad \text{ and } \quad
n\cdot (D+\beta \A)u=\beta g \quad \text{ on } \partial\Omega,
\end{equation*}
where $\beta=h^{-1}$.
By Lemma \ref{lemma-9-1}, the magnetic field $ \B$ satisfies the condition 
$$
\sup_{\b (x, r)} |\nabla  \B| \le \frac{C}{r} \fint_{\b (x, r)} | \B|
$$
for any $\b (x, r)\subset \widetilde{\Omega}$.
It follows that $\beta \B$ satisfies the same condition. Moreover, 
by the proof of Lemma \ref{lemma-9-1}, there exists $c>0$, depending only on $\Omega$,  $(\kappa, c_0)$ in \eqref{H-9} and
$\|\B \|_{C^{\kappa+1}(\overline{\Omega})}$ such that 
$$
\fint_{\b (x, r)} |\B|\ge cr^\kappa
$$
for any $x\in \overline{\Omega}$ and $0< r< r_0$. This implies that
\begin{equation}\label{low-b}
\sup_{\b (x, r)} |\beta \B|\ge c \beta r^\kappa.
\end{equation}
for any $x\in \overline{\Omega}$ and $0< r< r_0$, where $c>0$ depends only on $\Omega$,  $(\kappa, c_0)$ in \eqref{H-9} and
$\|\B \|_{C^{\kappa+1}(\overline{\Omega})}$. 

Next, we fix $x_0\in \partial\Omega$ and consider a family of  Lipschitz domains $\mathcal{O}_t=\Omega \cap \b(x_0, tr_0)$,
where $t\in (c, 2c)$ and $c>0$ is sufficiently small.
Let $R_0=2cr_0$.
Note that $\mathcal{O}_t \subset \b(x_0, R_0)$,
$\beta \B$ satisfies \eqref{cod-1} for any $\b(x, r) \subset \b(x_0, 4R_0)$. 
Moreover, by \eqref{low-b},
$$
\sup_{\b(x, R_0/2)} |\beta \B|
\ge c\beta (R_0/2)^\kappa > 4 R_0^{-2},
$$
if $\beta> \beta_0=C R_0^{-\kappa-2}$ and $C$ is sufficiently large.
This allows us to apply Theorem \ref{thm-n9-1} to obtain 
\begin{equation}\label{est-f1}
\| (v)^*_t \|_{L^2(\partial\mathcal{O}_t)}
+ \| m(x, \beta \B)^{1/2} v \|_{L^2(\mathcal{O}_t)}
\le C \| n\cdot (D+\beta \A) u \|_{L^2(\partial \mathcal{O}_t)}, 
\end{equation}
where $v(x)= |(D+\beta \A) u(x) | + m(x, \beta \B) |u(x)|$ and $(v)^*_t$ denotes the nontangential maximal function of
$v$ with respect to the domain $\mathcal{O}_t$.
We point  out that the constant $C$ in \eqref{est-f1} does not depend on $t$,
as the family of Lipschitz domains $\{ \mathcal{O}_t: t\in (c, 2c) \}$
possesses uniform Lipschitz characters.
It follows that
$$
\aligned
 & \int_{\b(x_0, cr_0)\cap \partial\Omega}
|\mathcal{N} (v) |^2 
+\int_{\b(x_0, cr_0)\cap \Omega} m(x, \beta \B)|v|^2\\
 & \quad\le C \int_{\b(x_0, r_0)\cap \partial\Omega} |\beta g|^2
+ C \int_{\Omega \cap \partial \b(x_0, tr_0)} | (D+\beta \A) u|^2
\endaligned
$$
for $t\in (c, 2c)$, where 
$$
\mathcal{N}(v) (x) =\sup \big\{ |v(y)|:  y\in \Omega,  |y-x|< M_0 \text{\rm dist} (y, \partial\Omega)
\text{ and } \text{\rm dist}(y, \partial\Omega)< cr_0\big \}
$$
for $x\in \partial\Omega$. By integrating the inequality above in $t\in (c, 2c)$, we  obtain
$$
\aligned
 & \int_{\b(x_0, cr_0)\cap \partial\Omega}
|\mathcal{N} (v) |^2 
+\int_{\b(x_0, cr_0)\cap \Omega} m(x, \beta \B)|v|^2\\
 & \quad\le C \int_{\b(x_0, r_0)\cap \partial\Omega} |\beta g|^2
+ C \int_{\Omega \cap \b(x_0, r_0)} | (D+\beta \A) u|^2.
\endaligned
$$
By  a covering argument, this gives
\begin{equation}\label{10-1a}
\aligned
 & \int_{ \partial\Omega}
|\mathcal{N} (v) |^2 
+\int_{\Omega_{cr_0}} m(x, \beta \B)|v|^2
 &\le C \int_{ \partial\Omega} |\beta g|^2
+ C \int_{\Omega } | (D+\beta \A) u|^2,
\endaligned
\end{equation}
where $\Omega_s =\{x \in \Omega: \text{\rm dist}(x, \partial\Omega)< s \}$.

To bound $v(x)$ for $x$ away from $\partial\Omega$, observe that if $\b(x, cr_0)\subset \Omega$, 
$$
\aligned
|v(x)| & \le C  \left(  \fint_{\b(x, cr_0/2)}  | v|^2 \right)^{1/2}\\
& \le C \left(\fint_{\b(x, cr_0/2)} |(D+\beta \A) u|^2 \right)^{1/2},
\endaligned
$$
where we have used the interior estimate \eqref{I-0}
for the first inequality and \eqref{M1-0} for the second.
It follows  that
$$
\int_{\partial\Omega}
|(v)^*|^2
\le C \int_{\partial \Omega} | \mathcal{N}(v)|^2
+ C \int_\Omega |(D+\beta \A)u|^2.
$$
This, together with \eqref{10-1a}, leads to
\begin{equation}\label{10-1b}
\aligned
 & \int_{ \partial\Omega}
| (v)^* |^2 
+\int_{\Omega_{cr_0}} m(x, \beta \B)|v|^2
 &\le C \int_{ \partial\Omega} |\beta g|^2
+ C \int_{\Omega } | (D+\beta \A) u|^2.
\endaligned
\end{equation}

To bound $m(x, \beta \B) |v|^2$ away from $\partial\Omega$, we let $\b(x_0, 2cr_0)\subset \Omega$ and apply 
Theorem \ref{thm-n9-1} to the domain $\mathcal{O}_t=\b(x_0, tr_0)$, where
$t\in (c, 2c)$.
This yields
$$
\int_{\b(x_0, cr_0)} m(x, \beta \B) |v|^2
\le C \int_{\partial\b(x_0,  tr_0)} |(D+\beta\A) u|^2.
$$
By integrating the inequality above in $t\in (c, 2c)$, we obtain 
$$
 \int_{\b(x_0, cr_0)} m(x, \beta \B) |v|^2
\le C \int_{\b(x_0,  2cr_0)} |(D+\beta\A) u|^2.
$$
Thus, by a covering argument,
$$
\int_{\Omega\setminus \Omega_{cr_0}} m(x, \beta \B) |v|^2
\le C \int_\Omega |(D+\beta \A) u|^2.
$$
In view of \eqref{10-1b}, we have proved that
\begin{equation}\label{10-1c}
\aligned
  \int_{ \partial\Omega}
| (v)^* |^2 
+\int_{\Omega} m(x, \beta \B)|v|^2
 &\le C \int_{ \partial\Omega} |\beta g|^2
+ C \int_{\Omega } | (D+\beta \A) u|^2\\
& \le C \int_{ \partial\Omega} |\beta g|^2,
\endaligned
\end{equation}
where we have used the energy estimate for the last step.

Finally, note that  
$$
\beta^{-1} v =|(hD+\A) u(x)| + h m(x, h^{-1} \B) |u(x)| = v_h,
$$
 where
$h=\beta^{-1}$.
 Let $h_0=\beta_0^{-1}$. 
 The desired estimates \eqref{main-e1}-\eqref{main-e2}  for $h \in (0, h_0)$  now follow from \eqref{10-1c}.
Note that by Remark \ref{re-9-1}, if $h\in (0, h_0)$,
\begin{equation*}
 c \sum_{|\alpha|\le \kappa}
h^{-\frac{1}{|\alpha|+2}} |\partial^\alpha \B(x)|^{\frac{1}{|\alpha|+2}}
\le m(x, h^{-1} \B)
\le C h^{-\frac12},
\end{equation*}
where  $C, c>0$ depend only on $\Omega$,  $(\kappa, c_0)$ in \eqref{H-9} and
$\|\B \|_{C^{\kappa+1}(\overline{\Omega})}$.

The case $d=2$ can be handled by the method of descending. 
Indeed, let $u\in H^1(\Omega; \C)$ be a weak solution of
 the Neumann problem \eqref{NP0} with boundary data $g\in L^2(\partial\Omega; \C)$
  in a two-dimensional Lipschitz domain 
$\Omega$. 
Define 
$\widetilde{\A}=(A_1(x_1, x_2), A_2(x_1, x_2), 0)$ and $v(x_1, x_2, x_3)=u(x_1, x_2)$.
Then 
$(hD +\widetilde{\A})^2 v=0$ in $\mathcal{O}$, where 
 $\mathcal{O}=\Omega \times (0, r_0)$ is a bounded  Lipschitz domain in $ \R^3$.
 Note that $n \cdot (h D+\widetilde{\A}) v=0$ on $\Omega \times \{ 0, r_0\} $, and that 
 $n \cdot (h D+\widetilde{\A}) v = (n_1, n_2) \cdot (h D+\A)u$ on $\partial\Omega \times (0, r_0)$.
  It follows from the case $d=3$ that 
 $$
 \aligned
 \| ( (hD+\A) u)^*\|_{L^2(\partial\Omega)}
 & \le  \|(( h D +\widetilde{\A}) v)^* \|_{L^2(\partial\mathcal{O})}\\
&  \le C \| n \cdot (h D +\widetilde{\A}) v \|_{L^2(\partial\mathcal{O})}
 =C \| g \|_{L^2(\partial\Omega)}.
 \endaligned
 $$
 Observe that if $x=(x_1, x_2, x_3)=(x^\prime, x_3)$, then 
 $m(x, \nabla \times \widetilde{\A}) \approx m (x^\prime, \nabla \times \A)$.
 The other estimates in Theorem \ref{main-thm-1} may be obtained in the same manner.
 We omit the details.
\end{proof}

\begin{proof}[Proof of Theorem \ref{main-thm-2}]

The proof for the case $d\ge 3$ is similar to that of Theorem \ref{main-thm-2}, using Theorem \ref{thm-d9-1}.
The case $d=2$ can also be treated by the method of descending.
Let $u\in H^1(\Omega; \C)$ be a weak solution of the Dirichlet problem \eqref{DP0} with boundary data $f\in H^1(\partial\Omega; \C)$
in a two-dimensional 
Lipschitz domain $\Omega$.  
Let $\widetilde{\A}$, $v$ and $\mathcal{O}$ be defined as in the proof of Theorem \ref{main-thm-1}.
Then $v\in H^1(\mathcal{O}; \C)$ is a weak solution of  $(h D+\widetilde{\A})^2 v=0$ in $\mathcal{O}$ with 
Dirichlet data $F$, where $F(x_1, x_2, x_3)= u(x_1, x_2)$  on $\Omega\times \{ 0, r_0\}$, 
and $F(x_1, x_2, x_3)=f(x_1, x_2)$ for $(x_1, x_2, x_3) \in \partial\Omega \times (0, r_0)$.
It is not hard to check that
$$
\aligned
 & \| T^h (F)\|_{L^2(\partial\Omega \times (0, r_0))} + h \| m(x, h^{-1} \nabla \times \widetilde{\A}) F \|_{L^2(\partial\Omega \times (0, r_0))}\\
 & \qquad\qquad
 =\| T^h (f)\|_{L^2(\partial\Omega)} + h \| m(x^\prime, h^{-1} \nabla \times \A) f \|_{L^2(\partial\Omega)}, 
\endaligned
$$
and 
$$
\aligned
 & \| T^h (F)\|_{L^2(\Omega \times \{ 0, r_0\} )} 
 + h \| m(x, h^{-1} \nabla \times \widetilde{\A}) F \|_{L^2(\Omega \times \{0, r_0\} )}\\
 & \qquad\qquad
 \le C \| (h D +\A) u \|_{L^2(\Omega)}  + C h  \| m(x^\prime, h^{-1} \nabla \times \A) u  \|_{L^2(\Omega)}\\
 & \qquad\qquad \le C \| (h D +\A) u \|_{L^2(\Omega)},
\endaligned
$$
where we have used \eqref{M1-0}. Moreover, 
$$
\aligned
\int_\Omega | (hD+\A) u|^2
 & \le h \int_{\partial \Omega} | n\cdot  (hD+\A) u| | u|\\
 &\le C \| (h D +\A) u\|_{L^2(\partial\Omega)}
 \| h m(x, h^{-1} \B) u \|_{L^2(\partial\Omega)}\\
 & \le \e \| (h D+\A)u\|^2_{L^2(\partial\Omega)}
 + C \e^{-1} \| h m (x, h^{-1} \B) f \|_{L^2(\partial\Omega)}^2, 
 \endaligned
$$
where we have used the fact $m(x, h^{-1} \B) \ge r_0^{-1}$ for any $x\in \overline{\Omega}$,  if $h\in (0, h_0)$.
As a result, we have proved that 
$$
\aligned
& \|  ((h D +\A) u)^* \|_{L^2(\partial\Omega)}
+ h \|  ( m(x, h^{-1} \B ) u)^* \|_{L^2(\partial\Omega)}\\
& \le C \| T^h f \|_{L^2(\partial\Omega)}
+  C ( 1+\e^{-1})  h \| m(x, h^{-1} \B) f \|_{L^2(\partial\Omega)}
+ C \e \| (h D+\A) u \|_{L^2(\partial\Omega)}.
\endaligned
$$
By choosing $\e$ small so that $C\e\le (1/2)$, we obtain \eqref{main-e3}.
The estimate \eqref{main-e4} follows in a similar manner.
\end{proof}

\medskip

\noindent{\bf Conflict of interest.}  The author declares that there is no conflict of interest.

\medskip

\noindent{\bf Data availability. }
 Data sharing not applicable to this article as no datasets were generated or analyzed during the current study.

 \bibliographystyle{amsplain}
 
\bibliography{S2025-2.bbl}

\medskip

\begin{flushleft}

\noindent\textsc{Zhongwei Shen, Institute for Theoretical Sciences, Westlake University,\\
No. 600 Dunyu Road, Xihu District, Hangzhou, Zhejiang  310030, P.R. China.}\\
\emph{E-mail address}: \texttt{shenzhongwei@westlake.edu.cn} \\

\end{flushleft}

\end{document}